\documentclass[10pt]{article}
\usepackage[english]{babel}
\usepackage{amsmath,amsthm, amssymb}
\usepackage{amsfonts}
\usepackage{comment}
\usepackage[all]{xy}
\usepackage{tikz}
\usepackage{enumerate}
\usepackage{thmtools}
\usepackage{thm-restate}

\usepackage{hyperref}
\usepackage{cleveref}

\newtheorem{thm}{Theorem}[section]
\newtheorem{cor}[thm]{Corollary}
\newtheorem{lem}[thm]{Lemma}
\newtheorem{prop}[thm]{Proposition}
\theoremstyle{definition}
\newtheorem{defn}[thm]{Definition}
\newtheorem{convention}[thm]{Convention}
\newtheorem{problem}[thm]{Problem}
\theoremstyle{remark}
\newtheorem{rem}[thm]{Remark}

\numberwithin{equation}{section}

\newcommand{\R}{\mathbb R}
\newcommand{\Z}{\mathbb Z}
\newcommand{\Q}{\mathbb Q}
\newcommand{\Num}{\mathtt{Num}}
\newcommand{\wfl}{word-number-list}
\newcommand{\ewfl}{encoded list}

\title{Efficient computations with counting functions\\ on free groups and free monoids}
\author{Tobias Hartnick, Alexey Talambutsa}%
\date{}%

\begin{document}

\maketitle

\begin{abstract}
We present efficient algorithms to decide whether two given counting functions on non-abelian free groups or monoids are at bounded distance from each other and to decide whether two given counting quasimorphisms on non-abelian free groups are cohomologous. We work in the multi-tape Turing machine model with non-constant time arithmetic operations. In the case of integer coefficients we construct an algorithm of linear space and time complexity (assuming that the rank is at least $3$ in the monoid case). In the case of rational coefficients we prove that the time complexity is $O(N\log N)$, where $N$ denotes the size of the input, i.e.\ it is as fast as addition of rational numbers (implemented using the Harvey--van der Hoeven algorithm for integer multiplication). These algorithms are based on our previous work which characterizes bounded counting functions.
\end{abstract}

\section{Introduction}

\subsection{From combinatorics of words to quasimorphisms}

The study of words over a finite alphabet $S$ is a central topic in many areas of mathematics and computer science, including algebra, combinatorics, dynamical systems, decision problems and many others. In particular, the combinatorics of such words have been studied intensively during the last 50 years in algebra and computer science -- see e.g. \cite{Sapir} and \cite{Lothaire} for two very different perspectives. Of fundamental importance for the combinatorics of words is the \emph{subword relation}: A word  $v = s_1\cdots s_l$ with $s_1, \dots, s_l \in S$ is called a \emph{subword} of a word $w = r_1 \cdots r_m$ with $r_1, \dots, r_m \in S$ provided there is some $j \in \{1, \dots, m-l+1\}$ such that $s_i  = r_{j+i} \text{ for all }i \in \{1, \dots, l\}$; we then call $\{j+1, \dots, j+l\}$ an \emph{occurence} of $v$ in $w$. Many famous problems concerning the combinatorics of words are related to this subword relation. For example, the question whether a word of length $N$ can be reconstructed from the set of its subwords of length up to $f(N)$ was solved in \cite{Lothaire} and independently by V.I.~Levenstein in \cite{Levenstein}, and the corresponding length was shown to be $f(N)=\lceil \frac{N+1}2 \rceil$. If one replaces the set of subwords by the \emph{multiset} of subwords, then the currently best known lower bound is $f(N)\leq\lfloor\frac{16}7\sqrt{N}\rfloor+5$, due to Krasikov and Roditty \cite{Krasikov-Roditty}.
 
The present article is concerned with a quantitative refinement of the subword relation: Given two words $v,w$ over $S$ we denote by $\rho_v(w)$ the number of (possibly overlapping) occurences of $v$ in $w$; thus by definition we have $\rho_v(w) > 0$ if and only if $v$ is a subword of $w$. The collection $S^*$ of all words over $S$ is a free monoid, and we can consider $\rho_v$ as a function $\rho_v: S^* \to \mathbb N_0$, called the \emph{$v$-counting functions}. If $v$ happens to be a single letter, then this counting function is actually a monoid homomorphism; in general it will only be a  \emph{quasimorphism} in the sense that
\[
\sup_{w_1, w_2 \in S^*}| \rho_v(w_1w_2) - \rho_v(w_1) - \rho_v(w_2)| < \infty.
\]
To summarize, the counting of subwords is a natural source of quasimorphisms on free monoids. 

\subsection{Computations in bounded cohomology of free groups}

While not much is known about general quasimorphisms on monoids, there is a well-developed theory of quasimorphisms on groups, since these are closely related to bounded cohomology \cite{Frigerio} and stable commutator length \cite{scl}. Note that if  $F_n$ is a free group of rank $n \geq 2$ with basis $S = \{a_1, \dots, a_n\}$, then we can identify elements of $F_n$ with reduced words over the extended alphabet $S^{\pm} = \{a_1, \dots, a_n, a_1^{-1}, \dots, a_n^{-1}\}$. This then allows us to define, for every $v \in F_n$, a corresponding $v$-counting function $\rho_v: F_n \to \mathbb N_0$. It was pointed out by Brooks \cite{Brooks} that the symmetrizations
\[\varphi_v: F_n \to \Z, \quad w \mapsto \rho_v(w) - \rho_v(w^{-1})\]
of these counting functions are quasimorphism on the free group; symmetrization is needed to deal with the effects of cancellations in free groups. In the sequel we refer to $\varphi_v$ as the
\emph{$v$-counting quasimorphism}. Finite linear combinations of these quasimorphisms are known as \emph{counting quasimorphisms} or sometimes \emph{quasimorphisms of finite type}. Similarly, finite linear combinations of $v$-counting functions are simply known as \emph{counting functions}.

Every quasimorphism $\varphi: F_n \to \mathbb R$ gives rise to a bounded $2$-cocycle $d\varphi$ on $F_n$ given by $d\varphi(x,y) := \varphi(y) - \varphi(xy) + \varphi(x)$ and hence defines a class in the second bounded cohomology $H^2_b(F_n; \mathbb R)$ (see \cite{Frigerio}). We say that two quasimorphisms are \emph{cohomologous} if they define the same bounded cohomology class. In particular, this is the case if they are at bounded distance from each other with respect to the $\ell^\infty$-norm.

It follows from our previous results of $\cite{HT1}$ that the question whether two given counting quasimorphisms (say, with coefficients in $\Z$ or $\Q$) are cohomologous is decidable; in fact, one can extract from \cite{HT1} an explicit algorithm which decides this question. This algorithm is actually sufficient for many applications, see e.g. the work of Hase \cite{Hase} for an application concerning the $\mathrm{Out}(F_n)$-action on $H^2_b(F_n; \mathbb R)$. However, while the algorithm sketched in \cite{HT1} is effective, it is certainly not efficient. On the contrary, the purpose of the present article is to provide an \emph{efficient} algorithm and to analyze its complexity. Our main result then reads as follows:
\begin{thm}\label{Main1} Let $n \geq 2$ and $\mathfrak N \in \{\Z, \Q\}$. Then there exists an algorithm which, given two counting quasimorphisms on $F_n$ with coefficients in $\mathfrak N$, decides whether they are cohomologous. Moreover, this algorithm can be implemented in such a way that its runtime is at most $O(N)$ if $\mathfrak N = \Z$ and at most $O(N \log N)$ if $\mathfrak N = \Q$, where $N$ denotes the size of the input data and the implied constants depend on $n$.
\end{thm}
We will provide a more precise formulation of Theorem \ref{Main1} in Corollary \ref{CorMainThm} below after clarifying the data structures which will be used to encode counting quasimorphisms. In particular, we will specify precisely what we mean by the ``size'' of the input data. As far as our model of computation is concerned, we will be working in the multi-tape Turing machine model throughout.
There are a couple of variants of Theorem \ref{Main1} worth mentioning
\begin{itemize}
\item Instead of deciding whether two given counting quasimorphisms are cohomologous we also decide with essentially the same runtime whether they are at bounded distance from each other (see Corollary \ref{CorMainThm}).
\item Instead of considering counting quasimorphisms we can also consider counting functions, i.e.\ linear combinations of $v$-counting functions. Our algorithms still applies, essentially with the same runtime (see Corollary \ref{CorMainThm}).
\item Instead of considering counting functions over the free group $F_n$ we can also consider counting functions over the free monoid $M_n$. The same runtime estimates as above also hold in this case, provided $n \geq 3$ (see Corollary \ref{CompMon2}).
\end{itemize}
The fact that the runtime is linear in the case of $\Z$-coefficients is connected to the fact that addition of integers can be implemented in linear time. On the contrary, it is currently not known whether addition of large rational numbers can be implemented in linear time. In view of the recent work of Harvey and van der Hoeven \cite{Harvey-vanderHoeven}, $T(N) := N \log N$ is currently the best known time complexity for addition of rational numbers of size at most $N$, if rational numbers are encoded as (possibly unreduced) mixed fractions (see the discussion in Appendix \ref{AppArithmetic}), and this is how this function enters into the proof of the theorem. The reader is invited to check that our results also holds for more general coefficient groups  $\mathfrak N \subset \mathbb R$ instead of $\Z$ or $\Q$, as long as addition, subtraction and comparison in $\mathfrak N$ can be carried out efficiently by a multi-tape Turing machine. 

\subsection{Towards applications}
Theorem \ref{Main1} can be seen as a starting point for efficient computations in the bounded cohomology of the free group or more precisely, for efficient computations in its subspace $H^2_{b, {\rm fin}}(F_2, \Q) \subset H^2_{b}(F_2, \R)$ generated by counting quasimorphisms with rational coefficients. This subspace is actually quite remarkable: As pointed out by Grigorchuk \cite{Grigorchuk} it is dense in $H^2_{b}(F_2, \R)$ for a suitable topology (namely the topology of pointwise convergence of homogeneous representatives, cf. \cite{Hartnick-Schweitzer}) and by a result of Schweitzer and the first author it is moreover invariant under the natural action of $\mathrm{Out}(F_n)$ and hence independent of the chosen basis  \cite{Hartnick-Schweitzer}. Nevertheless we have to admit that there are many explicit examples of classes in $H^2_{b}(F_2, \R)$ which are of \emph{infinite type} (i.e.\ not representable by counting quasimorphisms), and it is still a major open problem how to decide whether two such quasimorphisms are cohomologous. 

This latter question is actually relevant beyond the theory of free groups. Namely, if $G$ is an arbitrary group and $F \subset G$ is a finitely-generated free subgroup, then every bounded cohomology class in $G$ restricts to a bounded cohomology class of $F$. It turns out that for a large class of groups with weak negative curvature properties (namely, \emph{acylindrically hyperbolic groups} in the sense of \cite{Osin}), every $2$nd bounded cohomology class is uniquely determined by its restrictions to (hyperbolically embedded) free subgroups \cite{HSi}. This provides a strong motivation to look for algorithms which decide whether two given quasimorphisms on a free group are cohomologous. Theorem \ref{Main1} solves this problem for quasimorphisms of finite type, whereas the case of quasimorphisms of infinite type remains open for the moment.

\subsection{Organization of the article}
The goal of this article is to establish Theorem \ref{Main1} and all of the variants mentioned thereafter. It turns out that the case of counting functions over free monoids is notationally simpler to handle than the case of counting quasimorphisms over free groups, but uses essentially the same idea. In the body of this text we will thus first discuss algorithms which are concerned with counting functions over free monoids. An informal discussion will be given in Section \ref{Sec2}, and after formalizing the relevant concepts in Section \ref{Sec3}, the monoid version of Theorem \ref{Main1} (i.e.\ Corollary \ref{CompMon2}) will be established in Section \ref{Sec4} by constructing and analyzing the desired algorithm. In the final section (Section \ref{Sec5}) we will then explain, how this algorithm has to be modified in the group case - this requires a lot of additional notation, but very few additional ideas. In the appendix we discuss in detail the implementations of the arithmetic operations for our coefficient groups, since these crucially influence the runtime of our algorithm.

\medskip

\noindent \textbf{Notation:} In the sequel we denote by $\mathbb N = \{1, 2, \dots\}$ the set of natural numbers (starting from $1$) and by $\mathbb N_0 = \{0, 1, 2, \dots\}$ its union with $\{0\}$. The letters $\Z$ and $\Q$ denote the ring of integers and the field of rational numbers respectively. For real-valued functions $f,g : \mathbb N \to [0, \infty)$ we are going to use the standard Bachmann-Landau notation
\[
g(x) = O(f(x)) \quad : \Leftrightarrow \quad \limsup (g(x)/f(x))  < \infty.
\]
By a slight abuse of notation we use the same notation also for functions which are only defined on a cofinite subset of $\mathbb N$.\\

\noindent \textbf{Acknowledgments:} T.H. is grateful to Simons Foundation for supporting his visit to Steklov International Mathematical Center. A.T. is grateful to the Mathematics Department of Technion for supporting his visits. The work of A.T. was performed in Steklov International Mathematical Center and partially supported by the Ministry of Science and Higher Education of the Russian Federation (agreement no. 075-15-2019-1614).  A.T. also would like to thank M.N.~Vyalyi for valuable discussions concerning rational arithmetic algorithms.

\section{Counting functions on monoids}\label{Sec:MonInf}\label{Sec2}

Throughout this section let $n \geq 2$ be a fixed integer. We will discuss the equivalence problem for counting functions over the free monoid $M_n$ with coefficients in either $\Z$ or $\Q$ and describe an algorithm for its solution in an informal way.

\subsection{Representing counting functions}
Let $\mathtt{A} = \{\mathtt{a_1}, \dots, \mathtt{a_n}\}$ be a finite set of cardinality $n \geq 2$. We identify elements of the free monoid $M_n$ with finite words over $\mathtt{A}$ (including the empty word $\varepsilon$) and given $w \in M_n$ denote by $|w|$ the \emph{word length} of $w$, i.e.\ the number of letters of the word $w$. We also fix a coefficient group $\mathfrak N \in \{\Z, \Q\}$.

Given two words $v = s_1\cdots s_l$ and $w = r_1 \cdots r_m$ over $\mathtt{A}$ with $m \geq l \geq 0$ we denote by $\rho_v(w)$ the number of (possibly overlapping) occurrences of $v$ in $w$, i.e.
\[
\rho_v(w) = |\{j \in \{1, \dots, m-l+1\} \mid s_i  = r_{j+i} \text{ for all }i \in \{1, \dots, l\}\}|.
\]
If $m< l$, we set $\rho_v(w) := 0$ by convention. The function $\rho_v: M_n \to \mathbb N$ is then called the \emph{$v$-counting function} on $M_n$. With this definition we have $\rho_\varepsilon(w) = |w|$. By a \emph{counting function} with coefficients in $\mathfrak N$ we mean a function of the form
\[
f = \sum_{i=1}^N x_i \rho_{w_i} : M_n \to \mathfrak N
\]
where $N \in \mathbb N_0$, $w_1, \dots w_n \in M_n$ are words and $x_1, \dots, x_N \in \mathfrak N$ are coefficients. We say that two counting functions $f_1$ and $f_2$ are \emph{equivalent}, denoted $f_1 \sim f_2$ provided
\[
\|f_1 -f_2\|_\infty < \infty.
\]
In this section we are going to use two different encodings of counting functions over $M_n$ with coefficients in $\mathfrak N$, namely \wfl s and weighted trees. While weighted trees are useful for visualizations purposes, we will mostly describe our algorithms in terms of \wfl s, since these are more closely related to the data structures that we will use in the actual implementations of our algorithms below.

\subsubsection{Word-number-lists}

By a \emph{\wfl} we mean a list of the form $L = ((w_1, x_1), \dots, (w_N, x_N))$, where $N \in \mathbb N_0$ and $w_j \in M_n$ and $x_j \in \mathfrak N$ for all $j \in \{1, \dots, N\}$. The \emph{associated counting function} is
\begin{equation}
\rho_L = \sum_{i=1}^N x_i \rho_{w_i}.
\end{equation}
Thus every {\wfl} encodes a counting function, and by definition every counting function is encoded by a \wfl. However, this {\wfl} is not unique, since it may e.g.\ happen that some words repeat in the list.

\subsubsection{Weighted trees}

Following \cite{HT1} we visualize every {\wfl} by a finite weighted tree as follows: Denote by $T_n$ the right-Cayley tree of $M_n$ with respect to the free generating set $\mathtt{A}$. Thus, by definition, the vertex set of $T_n$ is given by $V(T_n) = M_n$, and two vertices $x$ and $y$ are joined by an edge if and only if there exists $i \in \{1, \dots, n\}$ such that $x = y\mathtt{a}_i$ or $y = x\mathtt{a}_i$. We consider $T_n$ as a rooted tree with root given by the empty word $\varepsilon$. Given a {\wfl} $L = ((w_1, x_1), \dots, (w_N, x_N)$ we denote by $T_L$ the subtree of $T_n$ given by the convex hull of the vertex set
\[
V(T_L) := \{w \in M_n \mid \exists\, j \in \{1, \dots, N\}: w = w_j \} \cup \{\varepsilon\}.
\]
We then define a \emph{weight function} $\alpha_L: V(T_L) \to \mathfrak N$ by
\[
\alpha_L(w) = \sum_{w_j = w} x_j.
\]
For example, the weighted tree $(T_L, \alpha_L)$ associated with the list
\[
L = ((\varepsilon, -1), (\mathtt{a}_2, 6), (\mathtt{a}_3, -1), (\mathtt{a}_1^2, 4), (\mathtt{a}_1\mathtt{a}_2, 4), (\mathtt{a}_1\mathtt{a}_3, 4), (\mathtt{a}_3\mathtt{a}_1, 1), (\mathtt{a}_3\mathtt{a}_2, 1), (\mathtt{a}_3^2, 1))
\]
is given by the following picture:
\begin{center}
\begin{tikzpicture}
  [scale=.8,every node/.style={circle,fill=black!10}]
  \node(0)   at (15,5) {$-1$};
  \node[fill=red!30] (1) at (12,3) {$0$};
  \node[fill=green!30] (2) at (15,3)  {$-6$};
  \node[fill=blue!30] (3) at (18,3)  {$-1$};
   \node[fill=red!30] (4) at (11,1) {$4$};
  \node[fill=green!30] (5) at (12,1)  {$4$};
  \node[fill=blue!30] (6) at (13,1)  {$4$};
     \node[fill=red!30] (7) at (17,1) {$1$};
  \node[fill=green!30] (8) at (18,1)  {$1$};
  \node[fill=blue!30] (9) at (19,1)  {$1$};
  \foreach \from/\to in {0/1, 0/2, 0/3, 1/4, 1/5, 1/6, 3/7, 3/8, 3/9}
    \draw (\from) -- (\to);
\end{tikzpicture}
\end{center}
Different {\wfl s} may give rise to the same tree, but all of these lists correspond to the same counting function, since
\[
f_L = \sum_{w \in V(T_L)} \alpha_L(w) \rho_w.
\]
In the sequel we will prefer to work with \wfl s rather than weighted trees, but we will occasionally use weighted trees to visualize some of our algorithms.


\subsection{The equivalence problem}

\subsubsection{Equivalent counting functions and equivalent \wfl s}

Recall that counting functions $f_1$ and $f_2$ are called \emph{equivalent} (denoted $f_1 \sim f_2$) if $f_1-f_2$ is a bounded function. For every $w \in M_n$ we have the obvious equivalences (see \cite{HT1})
\begin{equation}\label{ExtensionRelations}
\rho_w \sim \sum_{i=1}^n \rho_{\mathtt{a}_iw} \quad \text{and} \quad \rho_w \sim  \sum_{i=1}^n \rho_{w\mathtt{a}_i}.
\end{equation}
We refer to the two kind of basic equivalences in \eqref{ExtensionRelations} as  \emph{left-extension equivalences} and \emph{right-extension equivalences} respectively. It was established in \cite{HT1} that these basic equivalences span the space of all linear relations between equivalence classes of counting functions as $w$ ranges over all elements of $M_n$, but we will not need this fact here. Using left- and right-extensions relations we can transform counting functions into equivalent counting functions. We extend our notion of equivalence to \wfl s in the obvious way:
\begin{defn}
Let $L_1, L_2$ be two \wfl s. We say that $L_1$ and $L_2$ are
\begin{itemize}
\item \emph{strictly equivalent}, denoted $L_1 \approx L_2$, if $f_{L_1} = f_{L_2}$;
\item \emph{equivalent}, denoted $L_1 \sim L_2$, if $f_{L_1} - f_{L_2}$ is bounded. 
\end{itemize}
\end{defn}
We apply the same definitions also to finite weighted trees.

\subsubsection{Formulation of the main problem}
The main algorithmic problem concerning counting functions on free monoids that we want to solve in this article is as follows:
\begin{problem}[Equivalence problem]\label{Prob1} Given two \wfl s $L_1, L_2$, decide whether $L_1 \sim L_2$.
\end{problem}
Note that, given \wfl s $L_1, L_2$ it is easy to construct a {\wfl} $L_3$ with $f_{L_3} = f_{L_1} - f_{L_2}$. We may thus assume in Problem \ref{Prob1} that $L_2 = ()$ is the empty \wfl. To solve  Problem \ref{Prob1}, we need the notion of a minimal \wfl.
\begin{defn} \label{DefMinimal}
Let $N \in \mathbb N_0$ and let  $L = ((w_1, x_1), \dots, (w_N, x_N))$ be a \wfl.
\begin{enumerate}[(i)]
\item If $N \geq 1$, then the \emph{maximal depth} of $L$ is defined as
\[
\mathtt{depth}(L) := \max \{|w_j| \mid j = 1, \dots, N\};
\]
if $N=0$ we set $\mathtt{depth}(L) := -1$. We say that $L$ is of \emph{constant depth} if $|w_j| = \mathtt{depth}(L)$ for all $j = 1, \dots, N$. (With this definition, the empty list is of constant depth $-1$.)
\item $L$ is called \emph{minimal} if $\mathtt{depth}(L) = \min \{\mathtt{depth}(L') \mid L' \sim L \}$.
\end{enumerate}
\end{defn}
We employ similar terminology also for weighted trees.
\begin{rem}\label{Bottom} The following useful fact is immediate from the definition: If $L$ is a {\wfl} of maximal depth $\ell$ and $L'$ is a {\wfl} of maximal depth $<\ell$, then the concatenation $L \cup L'$ is minimal if and only if $L$ is minimal. In that sense, minimality of a {\wfl} depends only on the ``bottom level''.
\end{rem}
\begin{rem}\label{ProblemReducedToMin}
By definition, every {\wfl} $L$ is equivalent to some minimal {\wfl} $L'$, and we observe that
\[
L \sim () \iff \mathtt{depth}(L') = -1
\]

In order to solve Problem \ref{Prob1} 
it thus suffices to solve the following problem.
\end{rem}
\begin{problem}[Minimality problem]\label{MainProblem} Given a {\wfl} $L$, find a minimal {\wfl} $L'$ which is equivalent to $L$.
\end{problem}
For the remainder of this section we will thus focus on Problem \ref{MainProblem}.
\begin{rem}
We say that a {\wfl} $L = ((w_1, x_1), \dots, (w_N, x_N))$ 
is \emph{normalized} if the words $w_i$ are all distinct, ordered by length and ordered lexicographically within words of the same length, and if $x_j \neq 0$ 
for all $j \in \{1, \dots, N\}$. Clearly every {\wfl} is strictly equivalent to a normalized {\wfl}, 
and hence we will mainly study Problem \ref{MainProblem} for \emph{normalized} lists. We will see later that (for our chosen encoding) there is an efficient algorithm to normalize any given \wfl.
\end{rem}

\subsection{Related brotherhoods and minimality}
In order to discuss our algorithms we introduce the following terminology. We recall that $T_n$ denotes the right-Cayley tree of $M_n$ with respect to $\mathtt{A}$; we identify elements of $M_n$ with vertices of $T_n$.

\subsubsection{Brotherhoods}

Let $w \in M_n = V(T_n)$. The \emph{depth} of $w$ is defined as its distance from the root. The vertices on the geodesic between $w$ and the root (including $w$ and the root) are called the \emph{ancestors} of $w$.

From now on assume that $w$ is not the root. Then $w$ admits a unique ancestor $v={\rm Fa}(w)$ of distance $1$, which is called its \emph{father}. The vertices with the same father as $w$ are called its \emph{brothers} and their collection, the \emph{brotherhood} of $w$, is denoted by $\{\mathrm{Fa}(w) \ast \}$. Thus, by definition each brotherhood contains exactly $n$ elements.

The \emph{depth} of a brotherhood is defined as the depth of any of its $n$ elements. We say that two element $u,v$ of $M_n$ are \emph{related}, denoted $u \smile v$, if they  differ at most by their first letter. In this case we also say that the brotherhoods $\{ \mathrm{Fa}(u)\ast \}$ and $\{ \mathrm{Fa}(v)\ast \}$ are \emph{related}.

If $w = \mathtt{a}_1 \cdots \mathtt{a}_\ell \in M_n$ is of length $\ell \geq 2$, then the (possibly empty) subword $\mathtt{a_2} \cdots \mathtt{a_{\ell-1}}$ is called the \emph{stem} of $w$. If $B$ is a brotherhood then all elements of $B$ have the same stem $\mathrm{stem}(B)$, and two brotherhoods $B_1$ and $B_2$ of depth $\geq 2$ are related if and only if $\mathrm{stem}(B_1) = \mathrm{stem}(B_2)$. Given $u \in M_n$, there are precisely $n$ (pairwise related) brotherhoods $B_1, \dots, B_n$ with stem $u$, and up to reordering these are given by $B_i = \{ \mathtt{a}_i\mathtt{u} \ast \}$.


\subsubsection{Weighted brotherhoods}
If $B$ is a brotherhood, then a normalized {\wfl} $\mathcal B = ((w_1, x_1), \dots, (w_m, x_m))$ with $m \leq n$ 
is called a  \emph{weighted brotherhood} of type $B$ if the words $w_1, \dots, w_m$ all belong to $B$. Thus, we explicitly allow a weighted brotherhood to contain less than $n$ pairs and even to be empty; if $\mathcal B$ is a non-empty weighted brotherhood, then its type is uniquely determined and denoted by $\mathrm{type}(\mathcal B)$. By definition, every weighted brotherhood is a list of constant depth.

We say that two non-empty weighted brotherhoods $\mathcal B$ and $\mathcal B'$ are \emph{related} if their types are related; by convention we declare the empty weighted brotherhood to be related to every weighted brotherhood. A weighted brotherhood $\mathcal B = ((w_1, x_1), \dots, (w_m, x_m))$ is called \emph{constant} if it is either empty or if $m = n$ and ${x}_1= \dots = {x}_m$. Otherwise it is called \emph{non-constant}.

If $L = ((w_1, x_1), \dots, (w_N, x_N))$ is a normalized {\wfl} 
and $w \in M_n$, then we denote by $L_{\{w\ast\}}$ the normalized sublist of $L$ consisting of all pairs $(w_j, x_j)$ from $L$ with $w_j \in \{ w\ast \}$. By definition this is a (possibly empty) weighted brotherhood, and we refer to it as the \emph{weighted sub-brotherhood of $L$ of type $\{ w\ast \}$}.

\subsubsection{Unbalanced \wfl s and minimality}
\begin{defn}
A normalized {\wfl} $L$ of maximal depth $\ell$ is called \emph{unbalanced} if there exist two related weighted sub-brotherhoods $\mathcal B_1$ and $\mathcal B_2$ of $L$ of depth $\ell$ such that $\mathcal B_1$ is non-constant and $\mathcal B_2$ is empty, otherwise it is called \emph{balanced}.
\end{defn}
The following result was established in \cite[Theorem 4.2]{HT1} (in the language of weighted trees):
\begin{thm}\label{StoppingCriterion} Every unbalanced normalized {\wfl} is minimal.\qed
\end{thm}
Note that one can see immediately from the associated weighted tree whether a given normalized {\wfl} is unbalanced by only looking at its bottom level. For example, the following weighted tree (over $M_3$) represents an unbalanced list, since the non-constant weighted brotherhood labeled $(4,2,1)$ is related to the empty weighted sub-brotherhood whose father is the vertex labelled $-6$.
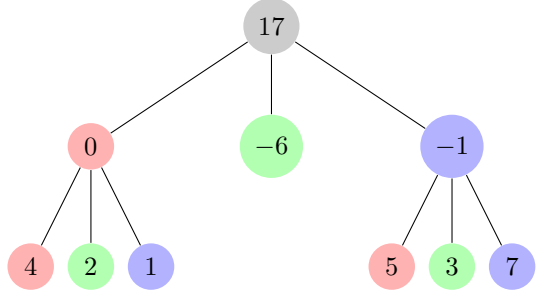
\begin{center}
\begin{tikzpicture}
  [scale=.8,auto=left,every node/.style={circle,fill=black!20}]
  \node(0) at (15,5) {$17$};
  \node[fill=red!30] (1) at (12,3) {$0$};
  \node[fill=green!30] (2) at (15,3)  {$-6$};
  \node[fill=blue!30] (3) at (18,3)  {$-1$};
   \node[fill=red!30] (4) at (11,1) {$4$};
  \node[fill=green!30] (5) at (12,1)  {$2$};
  \node[fill=blue!30] (6) at (13,1)  {$1$};
     \node[fill=red!30] (7) at (17,1) {$5$};
  \node[fill=green!30] (8) at (18,1)  {$3$};
  \node[fill=blue!30] (9) at (19,1)  {$7$};
  \foreach \from/\to in {0/1, 0/2, 0/3, 1/4, 1/5, 1/6, 3/7, 3/8, 3/9}
    \draw (\from) -- (\to);
\end{tikzpicture}
\end{center}

\subsection{Basic moves}
We now single out to basic moves which allow one to transform a given (normalized) {\wfl} into an equivalent one; these correspond to the two types of basic equivalences from \eqref{ExtensionRelations}.
\subsubsection{Pruning}
Let $L$ be a normalized {\wfl} of maximal depth $\ell \geq 1$ and let $B = \{ w \ast \}$ for some $w \in M_n$ of length $|w| = \ell-1$. We then say that $L_B$ is \emph{prunable} if it is constant and non-empty. This means that $L_B$ is of the form $L_B = ((w\mathtt{a}_1, x_1), \dots, (w\mathtt{a}_n, x_n))$ with $x_1 = \dots = x_n =: x$.
\begin{defn}
If $L_B$ is prunable, then the {\wfl} $\mathrm{Pr}_{w\ast}(L)$ obtained from $L$ by deleting $L_B$, appending the pair $(w,x)$ and normalizing the resulting list, is said to be obtained from $L$ by \emph{pruning} at $B$.
\end{defn}
It is immediate from the right-extension equivalence, that pruning transforms $L$ into an equivalent {\wfl}  $\mathrm{Pr}_{w\ast}(L)$. We say that a normalized {\wfl} is \emph{pruned} if it does not contain admit any prunable sub-brotherhood. Since pruning strictly reduces the number of entries in a given {\wfl}, every {\wfl} can be transformed into an equivalent \emph{pruned} {\wfl} by a finite number of pruning moves, as illustrated in the following example:\\

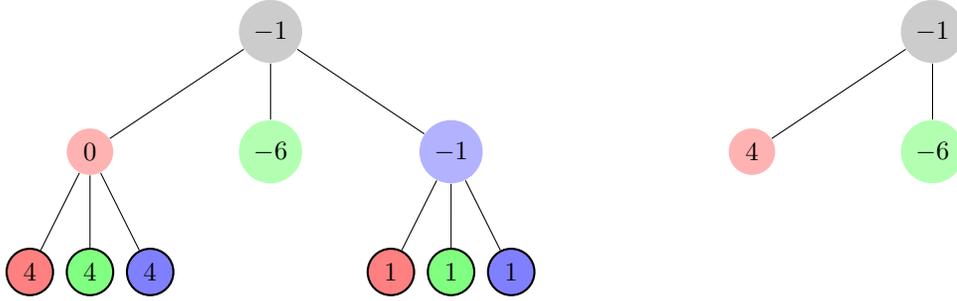
\begin{tikzpicture}
  [scale=.8,auto=left,every node/.style={circle,fill=black!20}]
  \node(0) at (5,5) {$-1$};
  \node[fill=red!30] (1) at (2,3) {$0$};
  \node[fill=green!30] (2) at (5,3)  {$-6$};
  \node[fill=blue!30] (3) at (8,3)  {$-1$};
   \node[fill=red!50,draw=black,thick] (4) at (1,1) {$4$};
  \node[fill=green!50,draw=black,thick] (5) at (2,1)  {$4$};
  \node[fill=blue!50,draw=black,thick] (6) at (3,1)  {$4$};
     \node[fill=red!50,draw=black,thick] (7) at (7,1) {$1$};
  \node[fill=green!50,draw=black,thick] (8) at (8,1)  {$1$};
  \node[fill=blue!50,draw=black,thick] (9) at (9,1)  {$1$};
   \node(0a) at (16,5) {$-1$};
  \node[fill=red!30] (1a) at (13,3) {$4$};
  \node[fill=green!30] (2a) at (16,3)  {$-6$};
  \foreach \from/\to in {0/1, 0/2, 0/3, 1/4, 1/5, 1/6, 3/7, 3/8, 3/9, 0a/1a, 0a/2a} 
    \draw (\from) -- (\to);
\end{tikzpicture}\\

In this example the pruned {\wfl} is already unbalanced and hence minimal by Theorem \ref{StoppingCriterion}, i.e.\ we have reached a minimal equivalent {\wfl} using only pruning moves. However, in general we will also need to apply moves related to left-extension-equivalences.

\subsubsection{Transfer}
We now consider the case where $L$ is a \emph{pruned} normalized balanced {\wfl} of maximal depth $\ell \geq 2$. Given a stem $u \in M_n$ we consider the $n$ related brotherhoods
$ \{\mathtt{a}_1 \mathtt{u}\ast\}, \dots,  \{\mathtt{a}_n \mathtt{u}\ast\}$ of depth $\ell$.
Since $L$ has maximal depth $\ell$, there exists some $u \in M_n$ with $|u| = \ell -2$ such that $L_{\{\ast u \ast\}} := L_{\{\mathtt{a}_1 \mathtt{u}\ast\}} \cup \dots \cup L_{\{\mathtt{a}_n \mathtt{u}\ast\}}$ is non-empty, and we fix such an element $u$.

We now form a $(n \times n)$-matrix $T(L, u)$ over $\mathfrak N$ whose entry $t_{ij}$ at position $(i,j)$ is given by $t_{ij} := x_{ij}$ if there is a pair of the form $(\mathtt{a}_i\mathtt{u} \mathtt{a}_j, x_{ij})$ in $L_{ \{\mathtt{a}_i \mathtt{u}\ast\}}$ and by $t_{ij} := 0$ otherwise. This matrix is called the \emph{transfer matrix} for the pair $(L, u)$. Note that $T(L,u)$ is non-zero by our choice of $u$. By definition we have
\[
\rho_{L_{\ast u \ast}} = \sum_{i=1}^n \sum_{j=1}^n t_{ij} \rho_{\mathtt{a}_i\mathtt{u} \mathtt{a}_j}.
\]
If we fix some $i' \in \{1, \dots, n\}$, then we can use the left-extension equivalence \eqref{ExtensionRelations} to rewrite this as
\begin{eqnarray*}
\label{NewCoefficients}
\rho_{L_{\ast u \ast}} &=&  \sum_{i=1}^n \sum_{j=1}^n (t_{ij}-t_{i'j}) \rho_{\mathtt{a}_i\mathtt{u} \mathtt{a}_j} +   \sum_{i=1}^n \sum_{j=1}^n t_{i'j} \rho_{\mathtt{a}_i\mathtt{u} \mathtt{a}_j}\\ &=&  \sum_{i\neq i'} \sum_{j=1}^n (t_{ij}-t_{i'j}) \rho_{\mathtt{a}_i\mathtt{u} \mathtt{a}_j} + \sum_{j=1}^n t_{i'j} \sum_{i=1}^n \rho_{\mathtt{a}_i\mathtt{u} \mathtt{a}_j}\\
&\sim& \sum_{i\neq i'} \sum_{j=1}^n (t_{ij}-t_{i'j}) \rho_{\mathtt{a}_i\mathtt{u} \mathtt{a}_j} +  \sum_{j=1}^n t_{i'j} \rho_{\mathtt{u} \mathtt{a}_j}.
\end{eqnarray*}
\begin{defn}
If $i' \in \{1,\dots, n\}$, then the list $\mathrm{Tr}_{i', u}(L)$ which is obtained from $L$ by deleting $L_{\ast u \ast}$, appending the pair $(\mathtt{a}_i\mathtt{u} \mathtt{a}_j, t_{ij} - t_{i'j})$ for every   $i \in \{1, \dots, n\} \setminus\{i'\}$ and $j \in \{1, \dots, n\}$, appending the pair $(\mathtt{u} \mathtt{a}_j, t_{i'j})$ for every $j \in \{1, \dots, n\}$  and normalizing the resulting list is said to be obtained from $L$ by \emph{transfer} of the brotherhood $L_{\{\mathtt{a}_{i'}u\ast \}}$.
\end{defn}
By the previous computation, $\mathrm{Tr}_{i', u}(L)$ is equivalent to $L$. The following picture shows the effect of a transfer move (applied to the weighted brotherhood labelled $(1,2,3)$) on the associated tree:
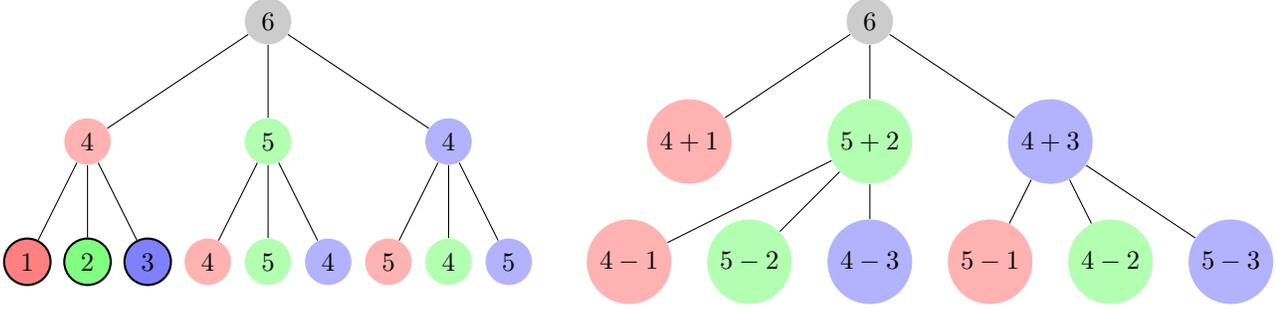
\begin{center}
\hspace*{-1.5cm}
\begin{tikzpicture}
  [scale=.8,auto=left,every node/.style={circle,fill=black!20}]
  \node (0) at (5,5) {$6$};
  \node[fill=red!30] (1) at (2,3) {$4$};
  \node[fill=green!30](2) at (5,3)  {$5$};
  \node[fill=blue!30] (3) at (8,3)  {$4$};
   \node[fill=red!50,draw=black!,thick] (4) at (1,1) {$1$};
  \node[fill=green!50,draw=black!,thick] (5) at (2,1)  {$2$};
  \node[fill=blue!50,draw=black!,thick] (6) at (3,1)  {$3$};
  \node[fill=red!30](20) at (4,1) {$4$};
  \node[fill=green!30](21)at (5,1)  {$5$};
  \node[fill=blue!30](22)at (6,1)   {$4$};
    \node[fill=red!30] (7) at (7,1) {$5$};
  \node[fill=green!30] (8) at (8,1)  {$4$};
  \node[fill=blue!30] (9) at (9,1)  {$5$};
    \node(10) at (15,5) {$6$};
  \node[fill=red!30] (11) at (12,3) {$4+1$};
  \node[fill=green!30] (12) at (15,3)  {$5+2$};
  \node[fill=blue!30] (13) at (18,3)  {$4+3$};
   \node[fill=red!30] (14) at (11,1) {$4-1$};
  \node[fill=green!30] (15) at (13,1)  {$5-2$};
  \node[fill=blue!30] (16) at (15,1)  {$4-3$};
    \node[fill=red!30] (17) at (17,1) {$5-1$};
  \node[fill=green!30] (18) at (19,1)  {$4-2$};
  \node[fill=blue!30] (19) at (21,1)  {$5-3$};
  \foreach \from/\to in {0/1, 0/2, 0/3, 1/4, 1/5, 1/6, 2/20, 2/21, 2/22, 3/7, 3/8, 3/9, 10/11, 10/12, 10/13, 13/17, 13/18, 13/19, 12/14, 12/15, 12/16}
    \draw (\from) -- (\to);
\end{tikzpicture}
\end{center}
\begin{defn}
A matrix $M = (m_{ij}) \in \mathfrak N^{n \times n}$ is called a \emph{column-row-sum} of a column vector $c=(c_1,\ldots,c_n)^\top \in \mathfrak N^n$ and a row vector $r=(r_1,\ldots,r_n)\in (\mathfrak N^{n})^*$, denoted $M=c \boxplus r$, if $m_{ij} = c_i+r_j$ for any $i,j\in \{ 1,2,\ldots, n\}$.
\end{defn}
Note that if $M = r \boxplus c$ is a column-row-sum, then for any two of its rows, say $m_i$ and $m_k$, the difference $m_i-m_k$ satisfies
\[
m_i - m_k = (c_i-c_k, \dots, c_i-c_k),
\]
hence is a constant vector. Conversely, if $M$ is a matrix with rows $m_1, \dots, m_k$, such that for all $i,k \in \{1, \dots, n\}$
 the difference $m_i-m_k$ is a constant vector with all entries equal to a constant $c_{ik}$, then $M$ is a column-row sum with
 \[
 M = (c_{1k}, \dots, c_{nk})\boxplus m_k.
 \]
\begin{prop}\label{MinimalityAfterTransfer} If the transfer matrix $T(L,u) \in \mathfrak N^{n \times n}$ is not a column-row-sum, then $L$ is minimal.
\end{prop}
\begin{proof} If $L' := \mathrm{Tr}_{i', u}(L)$, then $L'_{B_i'(u)}$ is empty. Now there are two cases: If all of the brotherhoods $L'_{B_i(u)}$ with $i \in \{1, \dots, n\}$ are constant, then for every $i \in \{1, \dots, n\}$ the difference $t_{ij}-t_{i'j}:=c_j$ is independent of $j$ so the matrix is a column-row-sum. If at least one of the weighted brotherhoods $L'_{B_i(u)}$ is non-constant (and in particular non-empty), then $L'$ has depth $\ell$ and is unbalanced, hence minimal by Theorem \ref{StoppingCriterion}. Since $L$ and $L'$ are equivalent and of the same depth $\ell$, we deduce that $L$ is also minimal. \end{proof}

\subsubsection{Transfer and prune}
We keep the previous notation; in particular, $L$ is a pruned normalized balanced {\wfl} of maximal depth $\ell \geq 2$ and $u \in M_n$ is chosen such that $|u| = \ell -2$ and $L_{\{\ast u \ast\}}$ is non-empty. 

If the transfer matrix $T(L, u)$ is non a column-row-sum, then we already known from Proposition \ref{MinimalityAfterTransfer} that $L$ is minimal. We thus consider the case where $T(L, u)$ is a column-row-sum and fix $i' \in \{1, \dots, n\}$. We then consider the list $L' := \mathrm{Tr}_{i', u}(L)$ obtained from $L$ by transferring the weighted brotherhood $L_{ \{\mathtt{a}_{i'} \mathtt{u}\ast\}}$.

Since $T(L, u)$ is a column-row-sum, all of the weighted sub-brotherhoods $L'_{ \{\mathtt{a}_{i'} \mathtt{u}\ast\}}$ for $i \in \{1, \dots, n\} \setminus\{i'\}$ are constant (by the proof of Proposition \ref{MinimalityAfterTransfer}), hence we can prune them to obtain a new equivalent list, which we denote by $\mathrm{TrP}_{i', u}(L)$ (for ``transfer and prune'').

Our next goal is to provide an explicit formula for the coefficients in $\mathrm{TrP}_{i', u}(L)$. By assumption, $T(L, u)$ can be written as a column-row sum $T(L, u) = c \boxplus r$, and we want to compute such a decomposition explicitly. For this it will be convenient to fix some $j' \in \{1, \dots, n\}$; since $T(L, u)$ is a column-row-sum we then have
\begin{equation}\label{DifferenceCondition}
{t}_{ij} - {t}_{i'j} = {t}_{ij'}  - {t}_{i'j'} \quad \text{for all }i,j \in \{1, \dots, n\}.
\end{equation}
Given $i,j \in \{1, \dots, n\}$ we now set
\begin{equation}\label{yz}
y_i := t_{ij'}-t_{i'j'}\quad \text{ and }\quad z_j := t_{i'j}
\end{equation}
so that
\[
T(L, u) = (y_1, \dots, y_n)^\top \boxplus (z_1,\dots, z_n).
\]
We have thus found the desired decomposition of $T(L, u)$. Note that $y_i$ is independent of the choice of $j'$ by \eqref{DifferenceCondition}. We then obtain
\begin{eqnarray*}
\rho_{L_{\ast u \ast}} &=& \sum_{i=1}^n \sum_{j=1}^n t_{ij} \rho_{\mathtt{a}_i \mathtt{u} \mathtt{a}_j}\\
&=& \sum_{i=1}^n \sum_{j=1}^n (t_{ij}-t_{i'j}) \rho_{\mathtt{a}_{i} \mathtt{u} \mathtt{a}_j} + \sum_{i=1}^n \sum_{j=1}^n t_{i'j} \rho_{\mathtt{a}_{i} \mathtt{u} \mathtt{a}_j} \\
&=& \sum_{i=1}^n (t_{ij'} - t_{i'j'}) \sum_{j=0}^n \rho_{\mathtt{a}_{i} \mathtt{u} \mathtt{a}_j}+\sum_{j=1}^n t_{i'j} \sum_{i=0}^n \rho_{\mathtt{a}_{i} \mathtt{u} \mathtt{a}_j}\\
&\sim& \sum_{i=1}^n y_i \rho_{\mathtt{a}_i \mathtt{u}} + \sum_{j=1}^n z_j \rho_{\mathtt{u} \mathtt{a}_j},
\end{eqnarray*}
and one can check that this rewriting corresponds precisely to applying transfer followed by pruning the resulting constant brotherhoods. This shows:
\begin{prop}\label{PropSupermove} The list $\mathrm{TrP}_{i', u}(L)$ is obtained from $L$ by removing $L_{\{\ast u \ast\}}$, appending the elements $(\texttt{a}_i \texttt{u}, y_i)$ and $(\texttt{u} \texttt{a}_j, z_j)$ for $i,j \in \{1, \dots, n\}$ and normalizing the resulting list.\qed
\end{prop}
\subsubsection{The case of maximal depth $\leq 1$}\label{SmallDepth}
Using the above moves we can reduce any given {\wfl} to a minimal {\wfl} or a normalized {\wfl} of maximal depth $\leq 1$. Thus assume from now on that $L$ is a normalized {\wfl} of maximal depth $\leq 1$; we are going to construct a minimal list $L'$ equivalent to $L$.

If $L$ is empty, then $L' := L$ is minimal by definition. Now assume $L$ has depth $0$; we claim that $L' := L$ is then minimal as well. Indeed, since $L$ is assumed normalized we have $\rho_L = x\rho_\varepsilon$ for some $x \neq 0$. Since $\rho_\varepsilon(w) = |w|$ for all $w \in M_n$, the function $\rho_L$ is unbounded, hence not equivalent to the $0$ function, and thus $L$ is not equivalent to the empty list, which is the unique list of depth $< 0$.

Finally, consider the case that $L$ has maximal depth $1$; then
\begin{equation}\label{rhoL1}
\rho_L = x \rho_{\varepsilon} + \sum_{i=1}^n x_i \rho_{\mathtt{a}_i}.
\end{equation}
If $L$ is prunable, i..e.\ $x_1 =\dots = x_n$, then applying a pruning move to the unique brotherhood of depth $1$ and normalizing yields either the empty list $L'$ (which is minimal by definition) or a list $L'$ of maximal depth $1$, which is minimal as seen above. We claim that, on the other hand, if a normalized {\wfl} $L$ of maximal depth $1$ is not prunable, then it is already minimal. Otherwise, with $\rho_L$ as in \eqref{rhoL1}, the function $\sum_{i=1}^n x_i \rho_{\mathtt{a}_i}$ would be equivalent to $y\rho_{\varepsilon}$ for some coefficient $y$, and hence the
function $\sum_{i=1}^{n} (x_i-y) \rho_{a_i}$ would be bounded. Since at least one of the coefficients $x_i - y$ is non-zero, this is impossible as can be seen by evaluating at words of the form $a_j^N$ for large $N$.
\subsection{Informal description of the algorithm}\label{AlgInformal}
Assume now that we are given a {\wfl} $L$. Then we can proceed as follows to find a minimal {\wfl} in the equivalence class of $L$:
\begin{itemize}
\item Normalize the {\wfl} $L$ and prune all constant brotherhoods in the bottom level to obtain a pruned {\wfl} $L'$. It $L'$ is unbalanced (or empty), then we have found our minimal {\wfl}.
\item  Assume now that $L'$ is balanced and of maximal depth $\ell \geq 2$. We can then find some $u \in M_n$ with $|u| = \ell -2$ such that $L'_{\{\ast u \ast\}}$ is non-empty. If the transfer matrix $T(L', u)$ is not a column-row-sum, then $L'$ is minimal and we are done. Otherwise we can apply a transfer and prune move. Repeat this as often as possible; if $L'$ is not minimal, then we ultimately obtain a list $L''$ of smaller depth.
\item We can now iterate the previous step until we obtain either a minimal list $L'''$ equivalent to $L$ or a list $L'''$ of depth $\leq 1$. In the latter case we can find an equivalent minimal list as described in Section \ref{SmallDepth}.
\end{itemize}
\begin{rem}
Unfortunately, this algorithm, if implemented naively, will not be efficient in all cases. Generally speaking, there are two things we want to avoid when applying a transfer-and-prune move. We do not want to create too many new entries in our list, and at the same time we do not want to transfer brotherhoods with large coefficients. We can easily avoid one of the two problems (by transferring the brotherhood with the minimal number of non-zero entries, respectively by transferring the brotherhood with ``smallest'' coefficients), but in general not both of them. In order to optimize the runtime of our algorithm we will apply a \emph{mixed strategy}:
\begin{enumerate}[(1)]
\item It the {\wfl} $L_{\{\ast u \ast\}}$ contains only few entries (compared to the maximal possible number of $n^2$ entries), then we transfer one of the weighted brotherhoods $L_{\{\mathtt{a}_iu\ast\}}$ which contains the fewest number of entries. (This will be referred to as the \textsc{sparse case}.)
\item If the {\wfl} $L_{\{\ast u \ast\}}$ contains many entries, then we transfer one of the brotherhoods $L_{\{\mathtt{a}_iu\ast\}}$ whose coefficients are as small as possible. (This will be referred to as the \textsc{non-sparse case}.)
\end{enumerate}
Using this strategy, we can ensure that the algorithm operates efficiently in all cases.
\end{rem}

\section{Formalization of the problem}\label{Sec3}

We now turn to the problem of formalizing the algorithms described in the previous section. Throughout this article we will use the multi-tape Turing machine model (see \cite[Section 8.4.1]{HopcroftEtAl}) as our computational model. We emphasize the fact that our algorithms are designed to work with coefficients that can be large, therefore we do not assume that the arithmetic operations with them are performed in constant time.

\subsection{Encoding the coefficients}
In order to formalize our algorithm we need to discuss how our data is stored inside a multi-tape Turing machine. In particular, we have to choose an encoding for our coefficients. The following result is established in Appendix \ref{AppArithmetic}. Here the function $T$ can be chosen as $T(N) := N \log N$ or any other function satisfying the assumptions of Convention \ref{ConvT}.

\begin{thm}\label{Encoding} For $\mathfrak N \in \{\Z, \Q\}$ there exist alphabets $\Sigma_{\mathfrak N}$,  encoding subsets $ \Num_{\mathfrak N} \subset \Sigma^*$ and surjective encoding maps
\[
\Num_{\mathfrak N} \to \mathfrak{N}, \quad \mathtt{x} \mapsto \langle \mathtt{x} \rangle
\]
and maps $\oplus, \ominus: \Num_{\mathfrak N} \times \Num_{\mathfrak N} \to \Num_{\mathfrak N}$ and $\|\cdot\|: \Num_{\mathfrak N} \to \mathbb N_0$ with the following properties:
\begin{enumerate}[(i)]
\item $\langle \mathtt{x}_1 \oplus \mathtt{x}_2 \rangle = \langle  \mathtt{x}_1 \rangle +  \langle  \mathtt{x}_2 \rangle$ and  $\langle \mathtt{x}_1 \ominus \mathtt{x}_2 \rangle = \langle  \mathtt{x}_1 \rangle -  \langle  \mathtt{x}_2 \rangle$.
\item If $|\mathtt{x}|$ denotes the word length of $\mathtt{x} \in  \Num_{\mathfrak N}$ as a word over $\Sigma_{\mathfrak N}$, then $|x| \leq \|x\| \leq 2 |x|$.
\item $\mathtt{x}_1 \oplus \mathtt{x}_2$ and $\mathtt{x}_1 \ominus \mathtt{x}_2$ are of size at most $\max\{\|\mathtt{x}_1\|,\|\mathtt{x}_2\|\}$ and can be computed in time $O(\max\{\|\mathtt{x}_1\|,\|\mathtt{x}_2\|\}+1)$ if $\mathfrak N = \Z$ and are of size at most $\| x_1 \| + \| x_2 \|$ and can be computed in time at most $O(T(\|\mathtt{x}_1\| + \|\mathtt{x}_2\|))$ if $\mathfrak N = \Q$.
\item Given $\mathtt{x}_1, \mathtt{x}_2 \in \Num_{\mathfrak N}$ it can be decided whether $\langle\mathtt{x}_1 \rangle = \langle \mathtt{x}_2 \rangle$ or not in time at most $O(\max\{\|\mathtt{x}_1\|, \|\mathtt{x}_1\|\})$ if $\mathfrak N = \Z$ and at most $O(T(\|\mathtt{x}_1\| + \|\mathtt{x}_2\|))$ if $\mathfrak N = \Q$.
\end{enumerate}
\end{thm}
In order to obtain the estimates in Theorem \ref{Encoding} in the case $\mathfrak N = \Q$ we had to chose an encoding that is not injective (namely an encoding by possibly non-reduced mixed fractions). We do not know whether similar bounds can be achieved using an injective encoding. Given $\mathtt{x}_1, \mathtt{x}_2 \in \Num_{\mathfrak N}$ we will write $\mathtt{x}_1 \equiv \mathtt{x}_2$ provided $\langle \mathtt{x}_1 \rangle = \langle \mathtt{x}_2 \rangle$.

From now on we assume that our coefficients and their arithmetic operations between them are encoded as in Theorem \ref{Encoding}. Given $x \in \Num_{\mathfrak N}$ we refer to $\|x\|$ as the \emph{size} of $x$, since it is proportional to the amount of memory needed in order to store $\mathtt{x}$.

\subsection{Encoding counting functions}
From now on we fix an integer $n \geq 2$. We want to encode counting functions over the free monoid $M_n$ with coefficients in either $\Z$ or $\Q$. In our informal discussion in Section~\ref{Sec:MonInf} we have described counting functions by \wfl s, i.e.\ lists of the form $L = ((w_1, x_1), \dots, (w_N, x_N))$, where $w_1, \dots, w_N$ are elements of $M_n$ and $x_1, \dots, x_N$ are elements of the coefficient group $\mathfrak N \in \{\Z, \Q\}$. In our formal discussion we will speak of the standard data structure of a \emph{doubly linked list} (see \cite[Section 10.2]{CormenEtAl}), which in multi-tape model can be presented as a long word written as on a separate tape with its list elements $(w_i,x_i)$ separated by commas, but we yet have to specify our encoding of the elements of $M_n$ and $\mathfrak N$ respectively. Since in most algorithms we deal with at most $n+1$ lists, this constant can serve as an estimate for a number of tapes one needs.

\subsubsection{Word-coefficient pairs}
To encode elements of $M_n$ we choose a set $\mathtt{A} := \{\mathtt{a}_1, \dots, \mathtt{a}_n\}$ of cardinality $n$ and identify $M_n$ with the set $\mathtt{A}^*$ of words over $\mathtt{A}$. This gives an encoding of $M_n$ over the alphabet $\mathtt{A}$, and given $w \in M_n$ we denote by $|w|$ the word length of $w$ with respect to the alphabet $\mathtt{A}$ and call it the \emph{size} of $w$. In reality, the amount of memory needed to encode $w$ is proportional to the canonical binary size $b(w)=\log_2(n) \cdot |w|$. However, we will primarily be interested in the case where the rank $n$ of our monoid is small compared to the size of the list and/or the size of the coefficients, hence we will treat $n$ as a constant throughout and thus consider $|w|$ to be proportional to the memory used by $w$.

For the coefficients we use the encoding $\Num_{\mathfrak N} \to \mathfrak N$ discussed in the previous section and in more details in Appendix~\ref{AppArithmetic}. Given $\mathtt{x} \in \Num_{\mathfrak N}$ we use the size $\|\mathtt{x}\|$ as defined in Appendix~\ref{AppArithmetic} as a measure for the memory needed to store $\mathtt{x}$.

By a \emph{word-coefficient-pair} (or simply a \emph{pair}) we shall always mean a pair of the form $(w, \mathtt{x}) \in \mathtt{A}^* \times \Num_{\mathfrak N}$. Here the first component $w$ is interpreted as an element of the monoid $M_n \cong \mathtt{A}^*$ and the second component $\mathtt{x}$ encodes a coefficient $x = \langle \mathtt{x} \rangle$. Given a pair $(w, \mathtt{x} )$ we define its \emph{total size} as
\[
|(w,\mathtt{x} )|_{\rm tot} := |w| + \|\mathtt{x} \|.
\]
By the discussion above, this quantity is proportional to the memory needed to store such a pair.

\subsubsection{Encoded lists}
A finite list $\mathcal L = ((w_1, \mathtt{x}_1), \dots, (w_N, \mathtt{x}_N))$ of pairs will be referred to as an \emph{\ewfl}, and the {\wfl}
 \[
 \langle \mathcal L \rangle := ((w_1, \langle \mathtt{x}_1\rangle), \dots, (w_N, \langle \mathtt{x}_N\rangle))
 \]
is called its \emph{interpretation}. Note that, due to the fact that our encoding of coefficients is not injective, different \ewfl s may share the same interpretation.  We say that \ewfl s are \emph{normalized, minimal, pruned, equivalent etc.} if their interpretations have the corresponding property. Given an {\ewfl} $\mathcal L =  ((w_1, \mathtt{x}_1), \dots, (w_N, \mathtt{x}_N))$, we set \[
|\mathcal L| := \sum_{i=1}^N |w_i|, \quad \|\mathcal L\| :=  \sum_{i=1}^N \|\mathtt{x}_i\| \quad \text{and} \quad |\mathcal L|_{\rm tot} := |\mathcal L| +  \|\mathcal L\|,
\]
and refer to these at the \emph{word size}, \emph{coefficient size} and \emph{total size} of $\mathcal L$ respectively. Up to a constant (depending on $n$) the total size of $\mathcal L$ is the amount of memory used to store this list in our encoding.

Given an {\ewfl} $\mathcal L$, the \emph{associated counting function} is defined as $\rho_{\mathcal L} := \rho_{\langle \mathcal L\rangle}$. For example, for $\mathfrak N = \Z$ the {\ewfl} $\mathcal L := (\mathtt{(a_1, +11), (a_2a_1, -100), (a_2,-101)})$ represents the counting function $3\rho_{a_1}-4\rho_{a_2 a_1}-5\rho_{a_2}$. Our choice of a doubly linked list as the underlying data structure (rather than, e.g.\ a structure similar to the weighted trees used for our visualizations above) is motivated by the fact that we do not
want to allocate redundant memory for many zero coefficients. We will see the efficiency of this data structure in the analysis of our main algorithm.

\subsection{Statement of the main result}

Having fixed the notion of an encoded list as our encoding for a counting function we can now formulate the main results of the present article, at least in the case of monoids.
\begin{thm}
\label{ComplexityMonoid} For every $n\geq 2$, there exists an algorithm \textsc{FindMinimalList} which takes as input an {\ewfl} $\mathcal L$ over $M_n$ and gives as output a minimal {\ewfl} $\mathcal M$ equivalent to $\mathcal L$. Moreover, if $n \geq 3$ the algorithm can be implemented in such a way that its time complexity is given as follows:
\begin{itemize}
\item[(a)] If $\mathfrak N = \Z$, then $\mathcal M$ is constructed from $\mathcal L$ in linear time, i.e. in time $O(|\mathcal L|)$.
\item[(b)] If $\mathfrak N = \Q$, then $\mathcal M$ is constructed from $\mathcal L$ in time $O(T(|\mathcal L|))$, where $T(N) := N \log N$.
\end{itemize}
\end{thm}
In view of Remark \ref{ProblemReducedToMin} we have the following immediate consequence:
\begin{cor}\label{CompMon2} For every $n \geq 2$ and $\mathfrak N \in \{\Z, \Q\}$ there exist algorithms to decide whether two counting functions over $M_n$ with coefficients in $\mathfrak N$ (given as \ewfl s) are equivalent. Moreover, for $n \geq 3$ these algorithm can be implemented in such a way that their respective time complexities are as decribed in Parts (a)--(b) of Theorem \ref{ComplexityMonoid}.\qed
\end{cor}
The raison d'\^etre for the appearance of the non-linear function $T$ in Theorem \ref{ComplexityMonoid}(b) (and consequently Corollary \ref{CompMon2})
is the non-linear complexity of the arithmetic operations over $\mathbb Q$, i.e.\ the currently best known implementation of addition in $\mathbb Q$ has time complexity $T(n)$. If addition of rational numbers could be implemented more efficiently, then this complexity bound could be improved; see Appendix \ref{AppArithmetic} for a more detailed discussion. While our algorithm works for all $n \geq 2$, our runtime estimate requires the more restrictive condition $n \geq 3$. This condition is only used once in Lemma \ref{SizeEstimationLemma} to show that in the main processing step the sizes of the coefficients under consideration is decreased by some fixed factor which is strictly less than $1$.

\subsection{Implementing basic moves}
Three of the main subalgorithms in the informal algorithm from Subsection \ref{AlgInformal} are given by pruning, computation of transfer matrices and the ``transfer and prune'' move. We now discuss how these algorithms can be implemented on the level of \ewfl s.

Concerning \ewfl s we will use similar terminology as for \wfl s. In particular, if $\mathcal L = ((w_1, \mathtt{x}_1), \dots, (w_N, \mathtt{x_N}))$ is a normalized {\ewfl} and $w \in M_n$, then we define the \emph{weighted sub-brotherhood of $\mathcal L$ of type $\{ w\ast \}$} as the normalized {\ewfl} $\mathcal L_{\{w\ast\}}$
consisting of all pairs $(w_j, \mathtt{x}_j)$ from $\mathcal L$ with $w_j \in \{w\ast\}$. Similarly, given a word $u \in M_n$ we denote by $\mathcal L_{\{\ast u \ast\}}$ the normalized {\ewfl} given as the concatenation of the \ewfl s $\mathcal L_{\{\mathtt{a}_1 \mathtt{u}\ast\}}$, \dots, $\mathcal L_{\{\mathtt{a}_n \mathtt{u}\ast\}}$.

\subsubsection{Pruning}

Let $\mathcal L$ be a normalized {\ewfl} with interpretation $L = \langle \mathcal L \rangle$ of maximal depth $\ell \geq 1$. If $w \in M_n$ is of length $\ell -1$ and $L_{\{w\ast\}}$ is prunable we would like to find an {\ewfl} which represents $\mathrm{Pr}_{w\ast}(L)$. For this we remove the sublist $\mathcal L_{\{w\ast\}}$ from $\mathcal L$ and append a pair of the form $(w,\mathtt{x})$ with $\mathtt{x} \equiv \mathtt{x}_1 \equiv \dots \equiv \mathtt{x}_n$.

Note that in performing such a pruning map to $\mathcal L$ we have the freedom of chosing $\mathtt{x}$. We could always choose $\mathtt{x} :=\mathtt{x}_1$, but in order to get an efficient algorithm it will be better to chose $\mathtt{x}$ to be one of the $\mathtt{x}_j$ of smallest size $\|\mathtt{x}_j\|$.

\subsubsection{Computation of transfer matrices}

Let $\mathcal L$ be a normalized {\ewfl} with interpretation  $L = \langle \mathcal L \rangle$ of maximal depth $\ell$ and let $u \in M_n$ with $|u| = \ell -2$. We will be interested in computing the transfer matrix $T := T(L, u) = (t_{ij})$. We say that a matrix $\mathtt{T} = (\mathtt{t}_{ij}) \in \Num_{\mathfrak N}^{n \times n}$  represents $T$ if $t_{ij} = \langle \mathtt{t}_{ij} \rangle$. In this case we also call $\mathtt{T}$ an \emph{encoded transfer matrix} and write $\mathtt{T}(\mathcal L, u)$ for $\mathtt{T}$.

In order to compute such a transfer matrix we first compute $\mathcal L_{\{\ast u \ast\}}$. We then read through the words in this list, and whenever we find a word starting in $\mathtt{a}_i$ and ending in $\mathtt{a}_j$, then we read out the corresponding coefficient $\mathtt{x}_{ij}$ and set $\mathtt{t}_{ij} := \mathtt{x}_{ij}$. All the other entries of $\mathtt{T}$ are set to be $\varepsilon$. We will see that this can be carried out in linear time in the total size of  $\mathcal L_{\{\ast u \ast\}}$.

\subsubsection{Transfer and prune}

Let $\mathcal L$ be a normalized {\ewfl} with interpretation $L = \langle \mathcal L \rangle$ of maximal depth $\ell \geq 2$. We assume that there is $u\in M_n$ with $|u| = \ell - 2$ such that the transfer matrix $T(L, u)$ is a column-row-sum. We then want to find an {\ewfl} $\mathcal L'$ which represents the {\wfl} $\mathrm{TrP}_{i', u}(L)$ for some $i' \in \{1, \dots, n\}$.

For this we first form the encoded transfer matrix $\mathtt{T}(\mathcal L, u)=(\mathtt{t}_{ij})$. For every $i \in \{1, \dots, n\}$ we then choose $\mathtt{y}_i \in \Num_{\mathfrak N}$ such that $\mathtt{y}_i \equiv \mathtt{t}_{ij}\ominus \mathtt{t}_{i'j}$ for some $j \in \{1, \dots, n\}$. For every $j \in \{1, \dots, n\}$ we then choose elements $\mathtt{z}_j \in \Num_{\mathfrak N}$ such that $\mathtt{z}_j \equiv t_{i'j}$. One possible choice is $\mathtt{z}_j := \mathtt{t}_{i'j}$ and $\mathtt{y}_{i} := \mathtt{t}_{ij'}\ominus\mathtt{t}_{i'j'}$ for some fixed $j' \in \{1, \dots, n\}$, but in general this choice may not be efficient. In any case, we can then modify $\mathcal L$ by removing  $\mathcal L_{\{\ast u \ast\}}$ and appending the elements $(\texttt{a}_i \texttt{u}, \mathtt{y}_i)$ and $(\texttt{u} \texttt{a}_j, \mathtt{z}_j)$ for $i,j \in \{1, \dots, n\}$ (where we may omit those with coefficient representing $0$). By Proposition \ref{PropSupermove} the resulting list will then represent $\mathrm{TrP}_{i', u}(L)$.

\section{Description of the algorithm in the monoid case}\label{SectionAlgorithm}\label{Sec4}
We now formalize the algorithm sketched in Section \ref{AlgInformal}. We then analyze its complexity and establish Theorem~\ref{ComplexityMonoid}. Our algorithm will work with \ewfl s of word-coefficient pairs. For brevity's sake we will refer to a word-coefficient pair simply as a \emph{pair} and to an {\ewfl} simply as a \emph{list}. We are going to deal with the integer and the rational case simultaneously. We set $T(N) := N$ if the coefficient group is given by $\mathfrak N = \Z$ and $T(N) := N \log N$ if the coefficient group is given by $\mathfrak N = \Q$ so that addition, subtraction and comparison of coefficients $\mathtt{x}_1, \mathtt{x}_2 \in \Num_{\mathfrak N}$ can be carried out in $T(\|\mathtt{x}_1\| + \|\mathtt{x}_2\|)$ by Theorem~\ref{Encoding}, and $T$ satisfies Properties (T1)-(T3) from Convention \ref{ConvT}.

\begin{rem}\label{SplittingLists}
During out main algorithm we will of often encounter algorithms of the following form: We are given an {\ewfl} $\mathcal L$, which we split into finitely many non-empty sublists $\mathcal L_1, \dots, \mathcal L_k$ by a procedure of linear time complexity $O(|\mathcal L|_{\rm tot})$. We then run the same procedure $\textsc{Proc}$ over each of the lists $\mathcal L_i$. 

Fortunately, we will always be in the situation where the time complexity of the procedure is either linear or of the form $O(T(N))$, where $N$ is the size of the input. In this specific situation it follows from Lemma \ref{CuttingLemma} that the time complexity of the whole algorithm is also of the form  $O(|\mathcal L|_{\rm tot})$ (in the linear case) or of the form $O(T(|\mathcal L|_{\rm tot}))$ respectively.
\end{rem}

\subsection{Normalizing lists and detaching brotherhoods}
We start the description of our algorithm by discussing some auxiliary procedures. We first consider a procedure to transform a given list into a normalized one.
\begin{lem} There exists a procedure \textsc{NormalizeList} with the following properties:
\label{NormalizeListLemma}
\begin{enumerate}[(i)]
\item The input of \textsc{NormalizeList} is a list $\mathcal L$ and the output is a a normalized list $\mathcal N$ equivalent to $\mathcal L$ with $|\mathcal N|_{\rm tot}\leq |\mathcal L|_{\rm tot}$.
\item The runtime of the procedure is $O(T(|\mathcal L|_{\rm tot}))$.
\end{enumerate}
\end{lem}
\begin{proof} Recall from \cite[Section 8.3]{CormenEtAl} that it is possible to sort a given list in linear time under assumption that the size of the alphabet is constant using the famous \textsc{RadixSort} sorting algorithm. (It is quite clear that this algorithm can be realized on a multi-tape Turing machine with $n+1$ tapes, where $n$ is the number of symbols in the alphabet of the words to be sorted.) In order to normalize the input list $\mathcal L$ we apply \textsc{RadixSort}, using the respective words as sorting keys. The result is a new list $\mathcal{L}'$ with entries being ordered in such a way that shorter words $w_i$ go first, and within words of the same length, words are ordered lexicographically. Obviously, such reordering does not change the total size of the list. We then go through the sorted list $\mathcal{L'}$, and if we find several consecutive pairs with the same word $w$ and coefficients $x_1, \dots, x_m$, then we replace these entries by a pair $(w, x_1 \oplus \dots \oplus x_{m})$ unless $x_1\oplus \dots \oplus x_m \equiv 0$ in which case we simply eliminate them. The result is a normalized list $\mathcal L$; the runtime of \textsc{RadixSort} is linear, hence (ii) follows from the Totalizing Lemma \ref{CuttingLemma} applied to the estimates for addition and comparison of  numbers from Theorem~\ref{Encoding}.

The inequality $|\mathcal N|_{\rm tot}\leq |\mathcal L|_{\rm tot}$ follows from the fact that in the course of the algorithm adjacent number-word pairs of the form $(w,\mathtt x)$ and $(w,\mathtt y)$ are replaced by the single pair $(w,\mathtt x \oplus \mathtt y)$, and that we have
\[
|(w,\mathtt x \oplus \mathtt y)|_{\rm tot} = |w|+||\mathtt{x} \oplus \mathtt{y}||< 2|w|+||\mathtt{x}||+||\mathtt{y}|| = |(w,\mathtt{x})|_{\rm tot}+|(w,\mathtt{y})|_{\rm tot}.\qedhere
\]
\end{proof}
One of the advantages of normalized lists is that it is easy to find sub-brotherhoods.
\begin{lem}
\label{detaching-lemma}
There exists a procedure \textsc{DetachBrotherhood} with the following properties:
\begin{enumerate}[(i)]
\item The input of \textsc{DetachBrotherhood} is a normalized list $\mathcal N$ and the procedure returns the first (in lexicographic order) sub-brotherhood $\mathcal B$ of $\mathcal N$ and deletes the sublist $\mathcal B$ from $\mathcal N$.
\item The runtime of the procedure is $O(|\mathcal B|)$.
\end{enumerate}

\end{lem}
\begin{proof} Just move the first pair $(\mathtt{a}_{i_1} \cdots \mathtt{a}_{i_l}, \mathtt{x})$ of $\mathcal N$ into a separate list $\mathcal B$; then read the word $w'$ of the next pair, and move the pair into $\mathcal B$ if $w' = \mathtt{a}_{i_1} \cdots  \mathtt{a}_{i_l-1}\mathtt{a}$ for some $\mathtt{a} \in \mathtt{A}$. Continue until you find a pair whose word is not of this form. Since we run through the list $\mathcal N$ until we find all entries of a non-empty list $\mathcal{B}$ plus we view at most the first $l+1$ letter of one more word, the runtime is $O(|\mathcal B|)$.
\end{proof}

\subsection{The Procedure ``\textsc{PruneList}''}
We now describe a procedure which takes a normalized list $\mathcal N$  of constant depth $\ell \geq 1$ and prunes all constant brotherhoods of $\mathcal N$. The result will be stored in two separate normalized lists: The remaining pairs of depth $\ell$ will be stored in a list $\mathcal N'$ and the newly produced pairs of depth $\ell-1$ will be stored in a list $\mathcal L$. The list $\mathcal L$ will be of much smaller size than $\mathcal N$, whereas the list $\mathcal N'$, which will be later be handled using a series of transfer-and-prune moves, may be of comparable size to $\mathcal N$ (and even equal to $\mathcal N$ if the latter is pruned to begin with).

\begin{lem}
\label{PCB-lemma}
There exists a procedure \textsc{PruneList} with the following properties:
\begin{enumerate}[(i)]
\item The input is a normalized list $\mathcal N$ of constant depth $\ell \geq 1$.
\item The output is a normalized sublist $\mathcal N'$ of $\mathcal N$ and another normalized list $\mathcal L$, which is either empty or of constant depth $\ell-1$.
\item The list $\mathcal N'$ is pruned, and $\mathcal N' \cup \mathcal L$ is equivalent to $\mathcal N$.
\item \label{size-property}
The size of the list $\mathcal L$ satisfies the inequality $|\mathcal L|_{\rm tot} \leq \frac{1}{n}(|\mathcal N|_{\rm tot}-|\mathcal N'|_{\rm tot})$
\item The runtime of the procedure is $O(T(|\mathcal N|_{\rm tot}))$.
\end{enumerate}
\end{lem}
Explicitly, such a procedure can be described as follows:
\bigskip
\hrule

\begin{center}
Procedure \textsc{PruneList}
\end{center}
\noindent \textsc{Input}: A normalized list $\mathcal N$ of constant depth $\ell \geq 1$.

\noindent \textsc{Output}: A normalized list $\mathcal N'$ of constant depth $\ell$ and a normalized list $\mathcal L$ of constant depth $\ell-1$ such that $\mathcal L$ is equivalent to the concatenation $\mathcal N' \cup \mathcal L$.

\begin{enumerate}
\item Set $\mathcal L$ and $\mathcal N'$ to be empty lists.
\item \label{repeated-step}
  While $\mathcal N$ is not empty do
\begin{enumerate}[(a)]
\item Apply the procedure \textsc{DetachBrotherhood} to $\mathcal L$, \\
  call the resulting sub-brotherhood $\mathcal B$ and its coefficients $\mathtt{x}_1, \dots, \mathtt{x}_k$\\
  and let the first word in $\mathcal B$ be of the form $w\mathtt{a}$ for some $\mathtt{a} \in \mathtt{A}$.
\item If $k \ne n$ or if $c_1 \equiv \dots \equiv c_n$ does not hold, \\
      then append $\mathcal B$ to $\mathcal N'$, \\
      else find the minimal $i_0 \in\{1,\ldots,n\}$ such that $||\mathtt{x}_{i_0}||=\min_{i=1,\ldots,n} ||\mathtt{x}_i||$ \\
      \phantom{else} and append $(w,\mathtt{x}_{i_0})$ to the list $\mathcal L$.
\end{enumerate}
\item Return the lists $\mathcal N'$ and $\mathcal L$.
\end{enumerate}
\hrule
\medskip
\begin{proof} It is clear that (ii) and (iii) hold, since $\mathcal N' \cup \mathcal L$ is obtained from $\mathcal N$ by pruning all prunable sub-brotherhoods of depth $\ell$. (Note that $\mathcal N'$ is normalized as a sublist of $\mathcal N$ and $\mathcal L$ is normalized by construction.) The list $\mathcal N'$ is obtained from $\mathcal N$ by deleting several sublists of the form \[\mathcal C = ((w\mathtt{a}_1,\mathtt{x}_1),(w\mathtt{a}_2, \mathtt{x}_2),\ldots,(w\mathtt{a}_n,\mathtt{c}_n)),\] and each time we delete such a sublist, we add the corresponding pair $(w, \mathtt{x}_{i_0})$ to $\mathcal L$. Since $|{w}| < |w\mathtt{a}_i|$ and $\|\mathtt{x}_{i_0}\| < \|\mathtt{x}_i\|$ for all $i \in \{1, \dots, n\}$, we have $|\{(w,\mathtt{x}_{i_0})\}|_{\mathrm{tot}}<\frac{1}{n}|\mathcal C|_{\mathrm{tot}}$, and thus we obtain (iv).

It remains to estimate the runtime of the procedure. To detach each brotherhood $\mathcal B$ we spend time $O(|\mathcal B|)$ according to Lemma \ref{detaching-lemma}. The comparison of $c_i$ and $c_{i-1}$ in Step \ref{repeated-step} takes $O(T(||c_i||+||c_{i-1}||))$ by Theorem~\ref{Encoding}, and we have to apply this $(n-1)$ times (for $i=2,\ldots,n$). Since $n$ is considered as a constant we need time $O(|\mathcal B|)+O(T(||\mathcal B||) = O(T(|\mathcal B|_{\mathrm{tot}})))$ to deal with the brotherhood $\mathcal B$. It thus follows from Lemma \ref{CuttingLemma} that the total time complexity is  $O(T(|\mathcal N|_{\rm tot}))$.
\end{proof}

\subsection{The procedure ``\textsc{TransferAndPrune}''}
\label{Sec:TransferAndPrune}
We now turn to the main step of our algorithm, in which we want to either show that a given pruned {\wfl} $L$ is minimal or otherwise apply a transfer-and-prune move. We will consider one family of related brotherhoods at a time, and hence our input will be a single collection of related non-constant weighted brotherhoods $(\mathcal B_1, \dots, \mathcal B_n)$. We may assume that the concatenation $\mathcal B = \mathcal B_1 \cup \dots \cup \mathcal B_n$ is normalized. This means that if ${u}$ denotes the common stem of the underlying brotherhoods, then the elements of $\mathcal B_i$ are of the form $(\mathtt{a}_i {u} \mathtt{a}_j, \mathtt{x}_{ij})$ with $\mathtt{x}_{ij} \not \equiv 0$. We recall that the encoded transfer matrix $\mathtt{T} = \mathtt{T}(\mathcal B, u)$ is given by $\mathtt{T} = (\mathtt{t}_{ij})$, where $\mathtt t_{ij} := \mathtt{x}_{ij}$ if $\mathcal B_i$ contains a pair with word $\mathtt{a}_iu\mathtt{a}_j$ and $\mathtt{t}_{ij} := \varepsilon$ otherwise. In particular, the total size of its entries is given by $\|\mathtt{T}\| = \|\mathcal B\|$ and we refer to the entries $\mathtt{t}_{ij}$ of $\mathtt{T}$ with $\mathtt{t}_{ij} \neq \varepsilon$ as the \emph{non-trivial} entries.

\begin{defn} We say that the encoded transfer matrix $\mathtt{T}$ is \emph{sparse} if either $n \geq 4$ and $\mathtt{T}$ has less than $3n$ non-trivial entries or if $n=3$ and $\mathtt{T}$ has less than $2n$ non-trivial entries.
\end{defn}

The goal of this subsection is to describe and analyze a procedure \textsc{TransferAndPrune} with the following properties:
\begin{enumerate}[(i)]
\item The input is an family $(\mathcal B_1, \dots, \mathcal B_n)$ of non-constant related brotherhoods of depth $\ell \geq 2$ such that $\mathcal B := \mathcal B_1 \cup \dots \cup \mathcal B_n$ is normalized.
\item The output is a Boolean variable $\textsc{minimal}$ and a list $\mathcal L$.
\item If $\textsc{minimal} = \texttt{true}$, then $\mathcal B$ is minimal (i.e.\ not equivalent to any list of depth $\leq \ell-1$) and $\mathcal L = \mathcal B$.
\item If $\textsc{minimal} = \texttt{false}$, then $\mathcal L$ is equivalent to $\mathcal B$ and of constant depth $\ell-1$.
\end{enumerate}
We will compute $\mathcal L$ from $\mathcal B$ by applying a single transfer-and-prune move if possible. The precise implementation of this move will very much depend on whether the transfer matrix is sparse or not. In the non-sparse case, a naive implementation of  the transfer-and-prune move works fine, but in the sparse case we need to take extra care in order not to create too many new word-coefficient pairs. Explicitly our procedure will look as follows:

\noindent\hrulefill
\negthickspace
\begin{center}
Procedure \textsc{TransferAndPrune}
\end{center}
\noindent \textsc{Input}: A family of non-constant weighted brotherhoods $(\mathcal B_1, \dots, \mathcal B_n)$ of depth $\ell \geq 2$\\
\phantom{Input:*} such that $\mathcal B := \mathcal B_1 \cup \dots \cup \mathcal B_n$ is normalized.

\noindent \textsc{Output}:  A Boolean variable \textsc{minimal} and a list $\mathcal L$.

\begin{enumerate}
\item \label{matrix-stem-construction}
Compute the stem $u$ of $(\mathcal B_1, \dots, \mathcal B_n)$ and the transfer matrix $\mathtt{T} := \mathtt{T}(\mathcal B, u)$.

\item \label{minimal-row-lookup}
Find the first row of minimal size in the matrix $\mathtt{T}$, \\ i.e. find the smallest element $i_0$ in the set $\{1, \dots, n\}$ such that
\[
\sum_{j=1}^n ||\mathtt{t}_{i_0j}|| = \min_{i=1, \dots, n} \sum_{j=1}^n ||\mathtt{t}_{ij}||.
\]
\item \label{balance-check}
For each $i\in\{1, \dots, n\} \setminus \{ i_0 \}$ do the following:
\begin{enumerate}[(a)]
\item If $ (\mathtt{t}_{i1} \ominus \mathtt{t}_{i_0 1}) \equiv (\mathtt{t}_{i2} \ominus \mathtt{t}_{i_0 2}) \equiv \dots \equiv (\mathtt{t}_{in}  \ominus \mathtt{t}_{i_0n})$ does not hold, then \\ return $\textsc{Minimal} := \texttt{true}$ and $\mathcal L := \mathcal B_1 \cup \dots \cup \mathcal B_n$  and terminate the procedure.

\end{enumerate}
\item Set $\mathcal L := ()$ and compute the number $K$ of non-trivial elements in the matrix $\mathtt{T}$. \\
 If ($n>3$ and $K<3n$) or ($n=3$ and $K<2n$) then set \textsc{Sparse}:=\texttt{True} \\
 \hspace*{5mm} else set \textsc{Sparse}:=\texttt{False}.

\item \label{non-sparse-case}
If (\textsc{Sparse}=\texttt{False}) do the following:
\begin{enumerate}[(a)]
\item For each $j\in \{1, \dots, n\}$, if $\mathtt{t}_{i_0j} \not \equiv 0$, append the pair $(u\mathtt{a}_j, \mathtt{t}_{i_0j})$ to $\mathcal L$.
\item Find the smallest element $j_0$ in the set $\{1, \dots, n\}$ such that
\[
(n-2)\cdot ||\mathtt{t}_{i_0j_0}|| + \sum_{i=1}^n ||\mathtt{t}_{ij_0}|| =  \min_{j=1, \dots, n} \left((n-2)\cdot ||\mathtt{t}_{i_0j}|| + \sum_{i=1}^n ||\mathtt{t}_{ij}||  \right).
\]
\item For each $i\in \{1, \dots,  n\}\setminus \{i_0\}$, if $\mathtt{t}_{ij_0} \not \equiv 0$, append $(a_iu, \mathtt{t}_{ij_0}\ominus \mathtt{t}_{i_0j_0})$ to $\mathcal L$.
\end{enumerate}
\item \label{sparse-case}
If (\textsc{Sparse}=\texttt{True}) do the following:
\begin{enumerate}[(a)]
\item Find the smallest element $i_1$ in $\{1, \dots, n\}$ with the property that the row $(\mathtt{t}_{i_11}, \dots, \mathtt{t}_{i_1n})$ of $\mathtt{T}$ has the minimal number of non-trivial entries and set\footnote{The letters $Z$ and $S$ stand for ``zero set'' and ``support'' respectively.}
\[
Z :=  \{j \in \{1, \dots, n\} \mid \mathtt{t}_{i_1j} = \varepsilon\} \quad  \text{and} \quad S:= \{1, \dots, n\} \setminus Z.
\]
\item For each $i\in \{1, \dots, n\}\setminus\{i_1\}$ do the following:
\begin{enumerate}[(i)]
\item Find the smallest element $j_1$ in $Z$ such that
\[
|\mathtt{t}_{ij_1}| = \min\{||\mathtt{t}_{ij}|| \, \mid j  \in Z\}.
\]
\item If $\mathtt{t}_{ij_1} \ne \varepsilon$, append $(a_iu, \mathtt{t}_{ij_1})$ to $\mathcal L$.
\end{enumerate}
\item If $i_0 = i_1$, then for each $j \in S$, if $\mathtt{t}_{i_1j} \ne \varepsilon$, append $(ua_{i_1}, \mathtt{t}_{i_1j})$ to $\mathcal L$.
\item If $i_0 \neq i_1$ do the following:
\begin{enumerate}[(i)]
\item Find the smallest element $j_0$ in $Z$ such that
\[
||\mathtt{t}_{i_0j_0}|| = \min\{||\mathtt{t}_{i_0j}||\, \mid j\ \in Z\}.
\]
\item For each $j \in S$, if  $ \mathtt{t}_{i_0j} \not \equiv \mathtt{t}_{i_0j_0}$ append $(ua_j, \mathtt{t}_{i_0j} \ominus \mathtt{t}_{i_0j_0})$ to $\mathcal L$.
\end{enumerate}
\end{enumerate}
\item Return $\textsc{Minimal} = \texttt{false}$ and $\mathcal L$.
\end{enumerate}
\negthickspace
\noindent \hrulefill

Let us first check correctness of the procedure:
\begin{prop} The procedure \textsc{TransferAndPrune} satisfies Properties (i) - (iv).
\end{prop}
\begin{proof} If the procedure stops during the execution of Step 3, then the encoded transfer matrix is not a column-row-sum and hence $\mathcal B$ is minimal by Proposition \ref{MinimalityAfterTransfer}, hence the output is correct. Assume now that the transfer matrix is a column-row-sum; we then want to perform a transfer-and-prune move. For this we have to choose some $i' \in \{1, \dots, n\}$ and then choose elements $\mathtt{y}_i, \mathtt{z}_j \in \Num_{\mathfrak N}$ with
\begin{equation}\label{Conditionsyz} \mathtt{z}_j \equiv t_{i'j}\quad \text{ and }\quad  \mathtt{y}_{i} \equiv \mathtt{t}_{ij'}-\mathtt{t}_{i'j'} \quad \text{for some } j'\in \{1, \dots, n\}.
\end{equation}
The list $\mathcal L$ should then consist of those pairs $(\texttt{a}_i \texttt{u}, \mathtt{y}_i)$ and $(\texttt{u} \texttt{a}_j, \mathtt{z}_j)$ with $\mathtt{y}_i \not \equiv 0 \not \equiv \mathtt{z}_j$. We distinguish two cases:

\textsc{Case 1.} If the matrix $\mathtt{T}$ is non-sparse, then in Step~\ref{non-sparse-case} we perform transfer with $i' := i_0$. The coefficients in our list $\mathcal L$ are then given by $\mathtt{z}_j := \mathtt{t}_{i_0j}$ and $y_i :=  \mathtt{t}_{ij_0}-\mathtt{t}_{i_0j_0}$, where $j_0$ is chosen in  Step~\ref{non-sparse-case}(b). (The reason for this specific choice will become clear in Lemma~\ref{SizeEstimationLemma}.) Then \eqref{Conditionsyz} holds by definition, hence the algorithm performs correctly.

\textsc{Case 2.} If $\mathtt{T}$ is sparse, then in Step \ref{sparse-case} we perform transfer with $i':= i_1$. For every $i$ we choose $\mathtt{y}_i := \mathtt{t}_{ij_1}$, where $j_1$ is chosen in Step \ref{sparse-case}(b)(i). Since $j_1 \in Z$ we have $\mathtt{t}_{i_1j_1} \equiv 0$ and hence $\mathtt{y}_i = \mathtt{t}_{ij_1} \equiv  \mathtt{t}_{ij_1} - \mathtt{t}_{i_1j_1}$ satisfies \eqref{Conditionsyz}.

For reasons of efficiency, the coefficients $\mathtt{z}_j$ will be chosen differently depending on whether $i_0 = i_1$ or not. If $i_0 = i_1$, then we choose $\mathtt{z}_j := \mathtt{t}_{i_1j}$, which is obviously correct. If $i_0 \neq i_1$ then we choose $\mathtt{z}_j := \mathtt{t}_{i_0j} \ominus \mathtt{t}_{i_0j_0}$, where $j_0$ is chosen as in Step \ref{sparse-case}(d)(i). Since $j_0 \in Z$ we have $\mathtt{t}_{i_1j_0} \equiv 0$, and since $\mathtt{T}$ is a row-colum-sum we have
\[
\mathtt{t}_{i_1j_0}\ominus \mathtt{t}_{i_0j_0} \equiv \mathtt{t}_{i_1j}\ominus \mathtt{t}_{i_0j} \implies \mathtt{z}_j = \mathtt{t}_{i_0j} \ominus \mathtt{t}_{i_0j_0} \equiv  \mathtt{t}_{i_1j} \ominus \mathtt{t}_{i_1j_0} \equiv \mathtt{t}_{i_1j}
\]
hence \eqref{Conditionsyz} is satisfied also in this case and the algorithm is correct.
\end{proof}

Concerning the run-time of the procedure \textsc{TransferAndPrune} we observe:
\begin{prop} The procedure \textsc{TransferAndPrune} terminates in time $O(T(|\mathcal B|_{\rm tot}))$.
\end{prop}
\begin{proof} We recall that the rank $n$ is considered as a constant throughout our estimates, hence all implied constants will be allowed to depend on $n$.

Step 1 can be perforemd in linear time $O(|\mathcal B|_{\rm tot})$. The stem can be computed from any pair $(w, \mathtt{x})$ in time $O(|w|) \leq O(|\mathcal B|_{\rm tot})$. The transfer matrix can also be computed in time $O(|\mathcal B|_{\rm tot})$ by running through the pairs and copying the required data.

Step 2 can be done by a subsequent summation of the length of the coefficients from $\mathcal B_{i}$ and the update of the minimum, so it takes time at most $O(||\mathcal B_i||)$ for each row, hence at most $O(||\mathcal B||)$ altogether. (Note that we are adding integers, not rational numbers, here, even in the case $\mathfrak N=\mathbb Q$.)

In Step 3 we need time at most  $O(T(||\mathtt{t}_{ij}||+||\mathtt{t}_{i_0j}||+||\mathtt{t}_{i {j-1}}||+||\mathtt{t}_{i_0{j-1}}||))$ to decide whether $\mathtt{t}_{ij} \ominus \mathtt{t}_{i_0 j}\equiv \mathtt{t}_{i{j-1}} \ominus \mathtt{t}_{i_0 {j-1}}$ by Theorem \ref{Encoding}. Then all of the comparisons in the $i$th iteration of Step 3 take time at most $O(T(2(||\mathcal B_i||+||\mathcal B_{i_0}||)))$. By our choice of $i_0$ and Property (T3) of the function $T$ (cf.\ Convention \ref{ConvT}) we have $O(T(2(||\mathcal B_i||+||\mathcal B_{i_0}||))) = O(T(||\mathcal B_i||)$, and hence by Lemma \ref{CuttingLemma} Step 3 takes time at most  $O(T(|\mathcal B|_{\rm tot}))$. Step 4 can be carried out in time $O(1)$.

In the non-sparse case we execute Step 5. Here Part (a) takes time at most $|\mathcal B_{i_0}|_{\rm tot}$, Part~(b) takes time at most $O(||\mathcal B||)$ similarly to Step 2, and Part (c) takes time at most $O(T(|\mathcal B|_{\rm tot}))$ since each operation of subtraction $x_{i j_0}\ominus x_{i_0 j_0}$ takes time at most $O(T(||\mathcal B_i||+||\mathcal B_{i_0}||)$. We can thus apply the Totalling Lemma~\ref{CuttingLemma} as in Step 3 to get the estimate of $O(T(||\mathcal B||))$ for the coefficients processing and $O(|\mathcal B|)$ for the words processing, obtaining at most $O(T(|\mathcal B|_{\rm tot}))$ in total. In the sparse case we execute Step 6, whose time complexity can be analyzed similarly to Step 5. Step 7 has again time complexity $O(1)$. Thus every step of
the algorithm takes time at most $O(T(|\mathcal B|_{\rm tot}))$, hence this bound also serves as a time estimate for the whole procedure.
\end{proof}
The crucial point about the procedure \textsc{TransferAndPrune} is that it reduces the size of the output by a fixed constant $<1$, unless the input was already minimal:
\begin{lem}\label{SizeEstimationLemma}
Assume that \textsc{TransferAndPrune} returns $\textsc{minimal} = \texttt{false}$ when applied to $(\mathcal B_1, \dots, \mathcal B_n)$. If $n \geq 3$, then the sizes of the output list $\mathcal L$ and the input list $\mathcal B :=\mathcal B_1 \cup \dots \cup \mathcal B_n$ are related by the inequality
\[
|\mathcal L|_{\rm tot} \quad \leq \quad  \frac{8}{9} \cdot |\mathcal B|_{\rm tot}.
\]
\end{lem}

\begin{proof} Since $|\mathcal L|_{\rm tot} = |\mathcal L| + ||\mathcal L||$ and  $|\mathcal B|_{\rm tot} = |\mathcal B|_{\texttt{A}} + ||\mathcal B||$ it will suffice to estimate the word lengths and coefficient size separately, i.e.\ to show that
\[
|\mathcal L| \leq \frac{8}{9}\cdot  |\mathcal B| \quad \text{and} \quad||\mathcal L|| \leq \frac{8}{9}\cdot  ||\mathcal B||.
\]
For this we will use the fact that if $m_1, \dots, m_n$ are rational numbers and $m_{i_0}$ is minimal among those, then
\begin{equation}\label{EstimateAverage}
m_{i_0} \leq \frac{1}{n} \sum_{i=1}^n m_i
\end{equation}
Let $\mathtt{T}$ be the coeffcient matrix constructed in Step \ref{matrix-stem-construction}. We will distinguish two cases:\\

\textsc{Case 1:} The matrix $\mathtt{T}$ is non-sparse.\\

Assume first that $n \geq 4$. so that  in our input we had at least $3n$ words if lenght $\ell$ and hence $|\mathcal B| \geq 3n\ell$. The maximal number of words in our output is $n + (n-1) = 2n-1$ (since in Step \ref{non-sparse-case}(a) we create at most $n$ words and in Step \ref{non-sparse-case}(c) we create at most $(n-1)$ words), and each of them has length $\ell-1$, hence
\[
|\mathcal L| \leq (2n-1)(\ell-1) < (2n-1) \cdot \ell \leq \frac{2n-1}{3n} |\mathcal B| \leq \frac {2}{3} |\mathcal B|  \leq \frac {8}{9}{|\mathcal B|}.
\]
For $n=3$ the same bound follows by a similar argument from
\[
{|\mathcal L|} \leq \frac{(2n-1)(\ell-1)}{2n\ell} |\mathcal B| \leq \frac{2n-1}{2n} |\mathcal B| = \frac{5}{6} |\mathcal B|  \leq \frac {8}{9}{|\mathcal B|}.
\]
On the other hand, $||\mathcal B||$ is precisely the total size of the entries in the matrix $\mathtt{T}$, i.e.
\begin{equation}\label{Matrix=Size}
||\mathcal B|| = \sum_{i=1}^n\sum_{j=1}^n ||\mathtt{t}_{ij}||,
\end{equation}
whereas in view of Step \ref{non-sparse-case}(a) and Step \ref{non-sparse-case}(c) the coefficient size of the output is
\[
||\mathcal L|| = \sum_{j=1}^n ||\mathtt{t}_{i_0j}|| + \sum_{i\neq i_0} ||\mathtt{t}_{ij_0}\ominus \mathtt{t}_{i_0j_0}||.
\]
In view of our choice of $i_0$ in Step \ref{minimal-row-lookup} we can apply \eqref{EstimateAverage} to estimate
\begin{equation}\label{minimal-row-norm}
\sum_{j=1}^n ||\mathtt{t}_{i_0j}|| \leq \frac{1}{n} \cdot \sum_{i=1}^n \left(\sum_{j=1}^n ||\mathtt{t}_{ij}||\right) = \frac{||\mathcal B||}{n}.
\end{equation}
Similarly, in view of our choice of $j_0$ in Step \ref{non-sparse-case}(b) we have the estimate
\[
(n-2)\cdot ||\mathtt{t}_{i_0j_0}|| + \sum_{i=1}^n ||\mathtt{t}_{ij_0}||  \leq \frac{1}{n} \cdot \sum_{j=1}^n \left((n-2)\cdot ||\mathtt{t}_{i_0j}|| + \sum_{i=1}^n ||\mathtt{t}_{ij}|| \right).
\]
Combining these two estimates and using Lemma \ref{CuttingLemma} we obtain
\begin{eqnarray*}
||\mathcal L|| &\leq& \frac{||\mathcal B||}{n} + \sum_{i\neq i_0} \left(||\mathtt{t}_{ij_0}|| + ||\mathtt{t}_{i_0j_0}||\right) \quad = \quad  \frac{||\mathcal B||}{n} + (n-2)\cdot ||\mathtt{t}_{i_0j_0}|| + \sum_{i=1}^n||\mathtt{t}_{ij_0}|| \\
&\leq& \frac{||\mathcal B||}{n} +  \frac{1}{n} \cdot \sum_{j=1}^n \left((n-2)\cdot ||\mathtt{t}_{i_0j}|| + \sum_{i=1}^n ||\mathtt{t}_{ij}|| \right)\\
&=&  \frac{||\mathcal B||}{n} +\frac{n-2}{n} \sum_{j=1}^n  ||\mathtt{t}_{i_0j}|| + \frac{1}{n }\cdot  \sum_{i=1}^n  \sum_{j=1}^n ||\mathtt{t}_{ij}||\\
&\leq& \frac{||\mathcal B||}{n} + \frac{n-2}{n} \cdot \frac{||\mathcal B||}{n} +\frac{1}{n} \cdot ||\mathcal B|| \quad = \quad \frac{3n-2}{n^2} \cdot ||\mathcal B||,
\end{eqnarray*}
hence for $n \geq 3$ we obtain
\[
||\mathcal L||  \leq \frac{7}{9}  \cdot ||\mathcal B|| \leq \frac{8}{9}  \cdot ||\mathcal B||,
\]
which finishes Case 1.\\

\textsc{Case 2:} The matrix $\mathtt{T}$ is sparse.\\

Let $\lambda$ be the minimal number of non-zero entries in a row of $\mathtt{T}$, i.e.\ the size of the set $S$ constructed in Step \ref{sparse-case}(a). Since the weighted brotherhoods $\mathcal B_1, \dots \mathcal B_n$ are non-constant we have $\lambda \geq 1$, and since $\mathtt{T}$ is sparse we have either $\lambda \leq 2$ and $n \geq 4$ or $\lambda \leq 1$ and $n = 3$. In either case we have $|Z| \geq 2$ and $\lambda \in \{1,2\}$. If $\lambda = 1$, then $S$ is actually a singleton, say  $S = \{j_s\}$. In this case we define
\[\label{AAc} A := \{i \in \{1, \dots, n\}\setminus\{ i_1 \} \mid \mathtt{t}_{ij_s} \equiv \mathtt{t}_{i_1j_s}\} \quad \text{and} \quad A^c := \{i \in \{1, \dots, n\}\setminus\{ i_1 \}\mid \mathtt{t}_{ij_s} \not \equiv \mathtt{t}_{i_1j_s}\}.\]
We then set $|A| := r$ so that $|A^c| = n-r-1$.

Assume first $\lambda = 2$. Then $\mathcal B$ contains at least $2n$ words and hence $|\mathcal B| \geq 2\ell$. In Step \ref{sparse-case}(b)(ii) we create at most $(n-1)$ words of length $(\ell-1)$ and in Step \ref{sparse-case}(c) or (d) we create at most $|S| = 2$ words. Thus for $\lambda = 2$ and $n \geq 3$ we obtain
\[
|\mathcal L| \leq (n+1)(\ell-1) < (n+1)\ell \leq \frac{n+1}{2n} |\mathcal B| \leq \frac{4}{6}|\mathcal B| \leq \frac{8}{9}|\mathcal B|.
\]
Now consider the case $\lambda = 1$. Since the matrix $\mathtt{T}$ is a column-row-sum, the $i$-th row of $X$ becomes constant after subtracting $\mathtt{t}_{i_1j_s}$ from $\mathtt{t}_{ij_s}$. Thus if $i \in A$, then $\mathcal B_i$ has a single entry, whereas if $i \in A^c$, then $\mathcal B_i$ has at least $(n-1)$-entries. We thus obtain
 \[
 |\mathcal B| \geq (1 + |A| + |A^c| \cdot (n-1))\cdot \ell \geq (1+r+2(n-r-1))\cdot \ell \geq (2n-r-1) \ell.
 \]
 In Step \ref{sparse-case}(b)(ii) we create $(n-r-1)$ words and in Step \ref{sparse-case}(c) or \ref{sparse-case}(d) we create at most $1$ word, all of length $\leq \ell-1$, hence
 \[
 |\mathcal L| \leq ((n-r-1)+1)\cdot(\ell-1) \leq (n-r) \cdot \ell \leq \frac{n-r}{2n-r-1} \cdot |\mathcal B|.
 \]
 If $r = 0$, then $ {(n-r)}/{(2n-r-1)} = n/(2n-1) \leq 3/5$, and if $r \geq 1$, then $2n-r-1 \leq 2(n-r)$, hence ${(n-r)}/{(2n-r-1)} \leq 1/2$, thus in any case
 \[
  |\mathcal L| \leq  \frac 8 9 \cdot |\mathcal B|.
 \]
We now turn to the coefficient sizes. Consider first the coefficients created in Step \ref{sparse-case}(b)(ii). For each $i\in \{1,2,\ldots,n\}\setminus \{i_1\}$ two cases are possible. Either no pair is created (if the $i$-th and $i_1$-th rows encode equal vectors) or a there are at least two non-trivial entries $\mathtt{t}_{ij_2}$ and $\mathtt{t}_{ij_3}$ in the $i$-the row with $\{j_2, j_3\} \subset Z$ (since $|Z| \geq 2$), and the smallest of these coefficients is copied to $\mathcal N$. Either way the coefficients created in the $i$th iteration of Step \ref{sparse-case}(b)(ii) are of total size at most $||\mathcal B_i||/2$ and consequently the total size of all coefficients created in Step \ref{sparse-case}(b)(ii) are of size at most $||\mathcal B||/2$.

We now consider the coefficients created in Parts (c) and (d) of Step \ref{sparse-case}. If $i_0 \neq i_1$, then their total size is given by
 \begin{eqnarray*}
\sum_{j \in S}||\mathtt{t}_{i_0j} \ominus \mathtt{t}_{i_0j_0}|| &\leq& \sum_{j \in S}\left(||\mathtt{t}_{i_0j}|| +||\mathtt{t}_{i_0j_0}||\right) \quad \leq \quad |S|||\mathtt{t}_{i_0j_0}|| +  \sum_{j \in S}||\mathtt{t}_{i_0j}||\\
&=& |S| \min_{j \in \{1, \dots, n\}} ||\mathtt{t}_{i_0j}|| +  \sum_{j \in S}||\mathtt{t}_{i_0j}|| \quad \leq \quad  |Z| \min_{j \in \{1, \dots, n\}} ||\mathtt{t}_{i_0j}|| +  \sum_{j \in S}||\mathtt{t}_{i_0j}||
\\ & \leq & \sum_{j \in Z}||\mathtt{t}_{i_0j}|| +  \sum_{j \in S}||\mathtt{t}_{i_0j}|| \quad \leq \quad \sum_{j =1}^n ||\mathtt{t}_{i_0j}|| \quad \leq \quad \frac{1}{n} \sum_{i=1}^n \sum_{j=1}^n  ||\mathtt{t}_{ij}|| \\& = & \frac{1}{n} \cdot ||\mathcal B||,
\end{eqnarray*}
where in the second line we have used that $|S| \leq 2 \leq |Z|$ and in the third line we have used \eqref{EstimateAverage}. If $i_0 = i_1$ then the choice for $i_0$ we made in the Step~\ref{minimal-row-lookup} and the estimation \eqref{minimal-row-norm} immediately shows that the coefficients created in Step \ref{sparse-case}(c) are bounded by $\frac{1}{n} \cdot ||\mathcal B||$. Either way, we see that for all $n\geq 3$,
\[
||\mathcal L|| \leq \frac{||\mathcal B||}{2} + \frac{||\mathcal B||}{n} = \frac{n+2}{2n} \cdot ||\mathcal B|| \leq \frac{5}{6} \cdot ||\mathcal B|| \leq \frac{8}{9} \cdot ||\mathcal B||.
\]
This finishes \textsc{Case 2}.
\end{proof}

\subsection{The main processing step}

In this section we describe the algorithm \textsc{MainProcessingStep} and then prove the Lemma~\ref{MPS-lemma} which is used afterwards to prove our main Theorem~\ref{ComplexityMonoid}.

\begin{lem}
\label{MPS-lemma} For every $n\geq 2$, there exists an algorithm \textsc{MainProcessingStep} with the following properties:
\begin{enumerate}[(i)]
\item The input is a normalized list $\mathcal L$ of constant depth $\ell \geq 2$.
\item The output is a boolean parameter $\textsc{minimal}$ and a normalized list $\mathcal L'$ equivalent to $\mathcal L$.
\item It $\textsc{minimal} = \mathtt{true}$, then $\mathcal L'$ is minimal.
\item If $\textsc{minimal} = \mathtt{false}$, then $\mathcal L'$ is of constants depth $\ell - 1$. If $n \geq 3$, then
\begin{equation}
|\mathcal L'|_{\mathrm{tot}} \leq \frac89|\mathcal L|_{\rm tot}.
\end{equation}
\item The runtime of the algorithm is $O(T(|\mathcal L|))$.
\end{enumerate}
\end{lem}
\noindent Explicitly we can describe such an algorithm as follows:\\

\noindent\hrulefill
\negthickspace
\label{MPS-procedure}
\begin{center}
Procedure \textsc{MainProcessingStep}
\end{center}
\noindent \textsc{Input} A normalized list $\mathcal N$ of constant depth $\ell \geq 2$.

\noindent \textsc{Output}: A Boolean variable \textsc{minimal} and a normalized list of pairs $\mathcal L$.

\begin{enumerate}
\item Set $\textsc{minimal} := \texttt{false}$.
\item Apply the procedure \textsc{PruneList} to $\mathcal N$ to create a  sublist $\mathcal N'$ of $\mathcal N$ and a list $\mathcal L$ of pairs of depth $\ell-1$.
\item Decompose the list $\mathcal N$ into lists $\mathcal A_1, \dots, \mathcal A_n$ by moving all the pairs $(w, x)$ in $\mathcal N$ such that $w$ starts with $a_i$ into $\mathcal A_i$.
\item While \textsc{minimal} = \texttt{false} and the lists $\mathcal A_1$, \dots, $\mathcal A_n$ are not all empty, do the following.
\begin{enumerate}[(a)]
\item For $i=1, \dots, n$, apply the procedure \textsc{DetachBrotherhood} to $\mathcal A_i$, call the resulting brotherhood $\mathcal B_i$.
\item If the brotherhoods $\mathcal B_1, \dots, \mathcal B_n$ are not related, set $\textsc{minimal} := \texttt{true}$, else apply the procedure \textsc{TransferAndPrune} to $(\mathcal B_1, \dots, \mathcal B_n)$.
Set its Boolean output as the new value of $\textsc{minimal}$ and append the resulting list to $\mathcal L$.
\end{enumerate}
\item Append $\mathcal A_1, \dots, \mathcal A_n$ to $\mathcal L$ and return $\textsc{NormalizeList}(\mathcal L)$.
\end{enumerate}
\negthickspace
\noindent \hrulefill

\begin{proof}[Proof of Lemma \ref{MPS-lemma}] Correctness of the procedure follows from the facts that the procedure \textsc{PruneList} correctly prunes all constant brotherhoods of maximal depth in $\mathcal N$ and that the procedure \textsc{TransferAndPrune} correctly applies a transfer-and-prune move if possible and otherwise sets the boolean variable $\textsc{minimal}$ to true.

By Lemma \ref{PCB-lemma}, Step 2 takes time $O(T(|\mathcal N|_{\rm tot}))$ to turn the initial list $\mathcal N$ turns into an equivalent list $\mathcal L \cup \mathcal N'$ such that \[0 \leq |\mathcal L|_{\rm tot} \leq \frac{1}{n}(|\mathcal N|_{\rm tot}-|\mathcal N'|_{\rm tot})\]  In Step 3 the list $\mathcal N'$ is split into the lists $\mathcal A_1, \mathcal A_2, \ldots, \mathcal A_n$ whose union is equal to $\mathcal N'$, and this takes time at most $O(|\mathcal N'|_{\rm tot})$.

Now we consider the iterations of Step 4. After each iteration the length of the list $\mathcal A_1\cup \mathcal A_2\cup \ldots\cup \mathcal A_n$ gets decreased by some integer value $|\mathcal B|_{\rm tot}$, where $\mathcal B$ is a union $\mathcal B_1\cup \ldots \cup B_k$, that was produced in the part (a). Each such step takes time $|\mathcal B|_{\rm tot}$ due to Lemma~\ref{detaching-lemma} and its output has size at most $\frac89 |\mathcal B|_{\rm tot}$ if $\textsc{minimal}=\texttt{False}$ after its execution. Using Lemma~\ref{CuttingLemma} we see that the time complexity of step $4$ is at most
\[
O(T(|\mathcal A_1 \cup \dots \cup \mathcal A_n|_{\mathrm{tot}})) \leq O(T(|\mathcal N'|_{\mathrm{tot}}))\leq O(T(|\mathcal N|_{\mathrm{tot}})).
\]
If $n \geq 3$  and $\textsc{minimal}=\texttt{False}$ at the end of Step $4$, then the original list $\mathcal L$ is increased into a list $\mathcal L'$ of total size at most
\[
|\mathcal L'|_{\mathrm{tot}} \leq |\mathcal L|_{\mathrm{tot}} + \frac{8}{9} |\mathcal N'|_{\mathrm{tot}} \leq \frac{1}{n} ( |\mathcal N|_{\mathrm{tot}} - |\mathcal N'|_{\mathrm{tot}} ) + \frac 8 9  |\mathcal N'|_{\mathrm{tot}} \leq \max\left\{\frac 1 n, \frac 8 9\right\}  |\mathcal N'|_{\mathrm{tot}} \leq \frac 8 9   |\mathcal N|_{\mathrm{tot}}.
\]
In any case, the list $\mathcal L'$ obtained by appending $\mathcal A_1, \dots, \mathcal A_n$ to the list obtained after Step $4$ always satisfies $|\mathcal L'|_{\mathrm{tot}} \leq |\mathcal N|_{\mathrm tot}$ (even if $n=2$ and/or $\textsc{minimal} = \texttt{true}$), and hence the final step can also be carried out in time at most $O(T(|\mathcal N|_{\mathrm{tot}}))$ and does not increase the size of the output. In the case, where $n=3$ and  $\textsc{minimal} = \texttt{false}$ the size of the output is thus smaller than the input by a factor of at least $\frac 8 9$.
\end{proof}

\subsection{The final algorithm}
Using all of the procedures described above we are now finally ready to describe an algorithm \textsc{FindMinimalList} which, given an arbitrary \ewfl, finds an equivalent minimal list.
This algorithm will then be used to prove Theorem~\ref{ComplexityMonoid}.\\

\noindent\hrulefill
\negthickspace
\begin{center}
Algorithm \textsc{FindMinimalList}
\end{center}
\noindent \textsc{Input} An {\ewfl} $\mathcal L$.

\noindent \textsc{Output}: A normalized  {\ewfl} $\mathcal M$ which is equivalent to $\mathcal L$.

\begin{enumerate}
\item Set $\mathcal N:=\textsc{NormalizeList}(\mathcal L)$. \label{NormalizationStep}
\item If $\mathcal N$ is empty, return $\mathcal M := \mathcal N$, otherwise set $\textsc{Minimal} := \texttt{false}$.\label{PreStep}
\item Set $d$ to be the maximal depth of $\mathcal N$.\label{DepthStep}
\item Decompose the list $\mathcal N$ into lists $\mathcal N_d, \dots, \mathcal N_1, \mathcal N_0$ by moving all pairs $(w, \mathtt x)$ in $\mathcal N$ with $|w|=i$ into $\mathcal N_i$.
    \label{DecompositionStep}
\item Set $\mathcal M_d:=\mathcal N_d$. \label{CopyStep}

\item For $i:=d$ downto 2 do \label{RunningStep}
\begin{enumerate}[(a)]
\item If $\mathcal M_i$ is empty, then set $\mathcal N'$ to be the empty list,\\
else apply \textsc{MainProcessingStep} to $\mathcal M_i$, \\ and put the resulting list into $\mathcal N'$ and the boolean value into $\textsc{Minimal}$.
\item If $\textsc{Minimal} = \texttt{false}$ then $\mathcal M_{i-1}:=\textsc{NormalizeList}(\mathcal N' \cup \mathcal N_{i-1})$ else break.
\end{enumerate}

\item If \textsc{minimal} = \texttt{true} return $\mathcal M := \mathcal M_i \cup \mathcal N_{i-1} \cup \ldots \cup \mathcal N_0$  else do the following. \label{FinalStep}
\begin{enumerate}[(a)]
\item Apply $\textsc{PruneList}$ to $\mathcal M_1$ to get a list $\mathcal N'$ of constant depth $1$ and a list $\mathcal K$ of constant depth $0$.
\item If $\mathcal N'$ is empty then return $\mathcal M:=\textsc{NormalizeList}(\mathcal K \cup \mathcal N_0)$ \\ else return $\mathcal M:=\mathcal M_1\cup \mathcal N_0$.
\end{enumerate}

\end{enumerate}
\negthickspace
\noindent \hrulefill

\begin{proof}[Proof of Theorem \ref{ComplexityMonoid}]
Basically we just have to combine all of the previous lemmas.

First, we show that the algorithm \textsc{FindMinimalList} gives a correct result. If $\mathcal N_0, \dots, \mathcal N_d$ denotes the lists created in Step~\ref{DecompositionStep}, then the list $\mathcal N_d \cup \mathcal{N}_{d-1} \cup \ldots \cup \mathcal N_1 \cup \mathcal N_0 = \mathcal N$ us equivalent to $\mathcal L$. Now during Step \ref{RunningStep} one of two cases occurs: Either the iteration never breaks and we obtains lists $\mathcal M_d, \dots, \mathcal M_1$ (in which case we set $t := 1$) or the algorithm produces lists $\mathcal M_d, \dots, \mathcal M_t$ for some $t \geq 2$ and breaks in the following iteration. In either case, one shows by a descending induction on $i$ that for all $i \in \{d, d-1, \dots, t\}$ the list $\mathcal M_i$ is a list of constant depth $\ell=i$ and equivalent to the union $\mathcal N_d \cup \ldots \cup \mathcal N_i$. Indeed, for $i := d$ there is nothing to show, and if the claim holds for $i$ and the algorithm does not break in the following step, then correctness of the \textsc{MainProcessingStep} ensures that $\mathcal M_{i-1}$ is of constant depth $i-1$ and satisfies
\[
\mathcal M_{i-1} \sim \mathcal N' \cup \mathcal N_{i-1} \sim \mathcal N_d \cup \dots \cup \mathcal N_i \cup \mathcal N_{i-1},
\]
which finishes the induction. Now, in Step~\ref{FinalStep} two cases are possible.

\emph{Case I.} If the execution of Step~\ref{RunningStep} a break occurs after computing $\mathcal M_d, \dots, \mathcal M_t$ for some $t \geq 2$, then correctness of the \textsc{MainProcessingStep} ensures that $\mathcal M_t$ is minimal and in fact unbalanced. This implies that also  $\mathcal M := \mathcal M_t \cup \mathcal N_{t-1} \cup \ldots \cup \mathcal N_0$ is unbalanced, hence minimal. Moreover, by the previous induction we have
\[
\mathcal M \sim \mathcal M_t \cup \mathcal N_{t-1} \cup \ldots \cup \mathcal N_0 \sim \mathcal N_d \cup \dots \cup \mathcal N_{t}  \cup \mathcal N_{t-1} \cup \ldots \cup \mathcal N_0 \sim \mathcal L
\]
Thus $\mathcal M$ is a minimal list, which is equivalent to $\mathcal L$ in this case.

\emph{Case II.} If all iterations of Step~\ref{RunningStep} were carried out without breaking, then we obtain a list $\mathcal M_1$ of constant depth $1$ which by the above induction is equivalent to  $\mathcal N_d \cup \ldots \cup \mathcal N_1$, and hence satisfies $\mathcal M_1 \cup \mathcal N_0 \sim \mathcal L$. If the list $\mathcal N'$ produced in Step \ref{FinalStep}(a) is non-empty, then $\mathcal M_1$ is a non-constant brotherhood of depth $1$, and by the argument in Section~\ref{SmallDepth} we deduce that $\mathcal M:=\mathcal M_1\cup \mathcal N_0$ is minimal, and thus a minimal list equivalent to $\mathcal L$. If, on the other hand, $\mathcal N'$ is empty, then $\mathcal M_1$ is equivalent to $\mathcal K$, and hence $\mathcal L \sim \mathcal M_1 \cup \mathcal N_0 \sim \mathcal K \cup \mathcal N_0$, i.e.\ $\mathcal L$ is equivalent to $\mathcal M:=\textsc{NormalizeList}(\mathcal K \cup \mathcal N_0)$. Moreover, $\mathcal M$ is a normalized list of maximal depth $\leq 0$, hence minimal. This finishes the proof of correctness in Case II.

It remains to estimate the time complexity of the algorithm; we proceed similarly to the complexity analysis in Lemma~\ref{MPS-lemma}. It suffices to show that each of the 7 steps of the algorithm is executed in time at most $O(T(|\mathcal N|))_{\rm tot}$. This is obvious for Step \ref{PreStep}. By Lemma~\ref{NormalizeListLemma}, Step~\ref{NormalizationStep} takes time $O(T(|\mathcal L|_{\rm tot}))$ and produces a list $\mathcal N$ with $|\mathcal N|_{\rm tot}\leq|\mathcal L|_{\rm tot}$. The latter implies in particular that Step \ref{DepthStep} can be performed in time $O(|\mathcal N|_{\rm tot})\leq O(|\mathcal L|_{\rm tot})$. Similarly, the split in Step~\ref{DecompositionStep} takes time $O(|\mathcal N|_{\rm tot})\leq O(|\mathcal L|_{\rm tot})$ since it can be done by a single run over the list $\mathcal N$. The complexity of Step~\ref{CopyStep} is also linear.

Now we consider the iterations of Step~\ref{RunningStep}. Let us assume that these iterations produce lists $\mathcal M_d, \dots, \mathcal M_t$ before an iteration breaks. We claim that
\begin{equation*}
\label{RecursiveInequality}
|\mathcal M_{i}|_{\rm tot} \leq \frac89|\mathcal  M_{i+1}|_{\rm tot} + |\mathcal N_i|_{\rm tot} \quad \text{for all } t \leq i \leq d-1.
\end{equation*}
Indeed, this inequality holds for trivial reasons if $\mathcal M_{i+1}$ is empty, and otherwise follows from Lemma~\ref{MPS-lemma}(iv). If we denote by $\mathcal M_{d+1}$ the empty list, then it also holds for $i = d$, since $|\mathcal M_d|_{\rm tot}=|\mathcal N_d|_{\rm tot}$. By descending induction on $i$ we then find that
\[
|\mathcal M_{i}|_{\rm tot} \leq \sum_{k=i}^d\left(\frac89\right)^{k-i}|\mathcal N_i|_{\rm tot} \quad (i \in \{t, \dots, d\}).
\]
Summing up the left and right parts over $i=t, \ldots, d$, we obtain, that
\[
\sum_{i=t}^{d}|\mathcal M_{i}|_{\rm tot} \leq \sum_{i=t}^{d}\left( \sum_{k=i}^d\left(\frac89\right)^{k-i}|\mathcal N_i|_{\rm tot}\right) \leq \sum_{i=t}^d \left( \sum_{j=0}^\infty \left(\frac89\right)^{j}\right)|\mathcal N_i|_{\rm tot}=  9\sum_{i=t}^d|\mathcal N_{i}|_{\rm tot}.
\]
Thus, if $\mathcal M_d, \dots, \mathcal M_t$ are produced before an iteration breaks, then
\begin{equation}
\label{MainInequality}
\sum_{i=t}^{d}|\mathcal M_i|_{\rm tot} \leq 9\sum_{i=t}^d|\mathcal N_{i}|_{\rm tot}\leq 9 |\mathcal N|_{\rm tot}.
\end{equation}

Given $i \in \{t, \dots, d\}$, the $(d-i+1)$th iteration of Step~\ref{RunningStep}(a) takes time $O(T(|\mathcal M_i|_{\rm tot})))$ by Lemma~\ref{MPS-lemma}(v). It thus follows from Lemma~\ref{CuttingLemma} that all of the iterations of Step~\ref{RunningStep}(a) taken together take time at most $O(|\mathcal M_t| + \dots + O(|\mathcal M_d|))$, and thus times at most $O(T(|\mathcal N|_{\rm tot})))$ due to Inequality \eqref{MainInequality} and the fact that the function $T$ satisfies Property (T3) from Convention \ref{ConvT}.

Similarly, for every $i \in \{t, \dots, d\}$, the call of the procedure $\textsc{NormalizeList}$ in the $(d-i+1)$-th iteration of Step~\ref{RunningStep}(b) takes time at most $O\left(T\left(\frac89|\mathcal M_{i}|_{\rm tot}+|\mathcal N_{i-1}|_{\rm tot}\right)\right)$. Summing over $i$ and using the Totalling Lemma~\ref{CuttingLemma} we obtain that together all these calls take time at most $O(T(\frac89|\mathcal M_t|_{\rm tot} + \dots + \frac89|\mathcal M_d|_{\rm tot}  +|\mathcal N|_{\rm tot}))$, which is again $O(T(|\mathcal N|_{\rm tot})$ by Inequality \eqref{MainInequality} and Property (T3). We have thus established that the whole Step 6 can be carried out in time $O((|\mathcal N|_{\rm tot})$.

Now assume that $\mathcal M_d, \dots, \mathcal M_t$ have been produced in Step~\ref{RunningStep} before a break. If $t \geq 2$, i.e.\ a break occured, then in Step~\ref{FinalStep} the algorithm just returns a minimal list $\mathcal M := \mathcal M_t \cup \mathcal N_{t-1} \cup \ldots \cup \mathcal N_0$ of size at most
\[
|\mathcal M|_{\rm tot} =  |\mathcal M_t|_{\rm tot} + \sum_{i=0}^{t-1} |\mathcal N_i|_{\rm tot}\leq   9\sum_{i=t}^{d}|\mathcal N_i|_{\rm tot} +  9 \sum_{i=0}^{t-1}|\mathcal N_i|_{\rm tot} =  9 |\mathcal N|_{\rm tot}.
\]
If $t=1$, i.e.\ no break occurs, then Step~\ref{FinalStep} initially operates with the list $\mathcal M_1$, which has size at most $m :=9 (|\mathcal N|_{\rm tot}-|\mathcal N_0|_{\rm tot})$ according to \eqref{MainInequality}. The pruning in Step \ref{FinalStep}(a) then takes time $O(T(|\mathcal M_1|_{\rm tot})) = O(T(m)) =  O((|\mathcal N|_{\rm tot})$ to produce a list $\mathcal N' \cup \mathcal K$ of size at most $m$ by Lemma \ref{PCB-lemma}. The potentially necessary normalization in Step \ref{FinalStep}(b) then takes time $O(T(m)) =  O((|\mathcal N|_{\rm tot})$ as well, and hence the whole Step 7 can be accomplished in time $O((|\mathcal N|_{\rm tot})$ and produces a list $\mathcal M$ of total size at most $9 |\mathcal N|_{\rm tot}$.
\end{proof}

\section{The group case}\label{Sec5}

In this section we explain how the techniques of the previous sections have to be modified to deal with counting functions on a free group rather than on a few monoid. We will see, that only at very few places some modifications will be necessary, although the notation gets more involved.

\subsection{Encoding counting functions on free groups}
In this section, we consider again a finite alphabet $\mathtt{A} = \{\mathtt{a}_1, \dots, \mathtt{a}_n\}$ of size $n \geq 2$. We then define the \emph{extended alphabet} $\mathtt{A}^\pm := \{\mathtt{a}_1, \dots, \mathtt{a}_n, \mathtt{a}^{-1}_1, \dots, \mathtt{a}^{-1}_n\}$, where $\mathtt{a}^{-1}_1, \dots, \mathtt{a}^{-1}_n$ are further symbols chosen such that $|\mathtt{A}^\pm| = 2n$. We then identify $F_n$ with the subset $\mathtt{A}^{\pm \ast} \subset (\mathtt{A}^\pm)^*$ consisting of words which are reduced, i.e.\ which do not contain any subword of the form $\mathtt{a}_j\mathtt{a}^{-1}_j$ or $\mathtt{a}^{-1}_j\mathtt{a}_j$ for some $j \in \{1, \dots, n\}$. Given a word $w \in \mathtt{A}^{\pm \ast}$ we denote by $w_{\mathrm{in}} \in \mathtt{A}^\pm$ and $w_{\mathrm{fin}} \in \mathtt{A}^\pm$ its initial and final letter respectively.

From now on let $\mathfrak N \in \{\Z, \Q\}$. By a \emph{word-coefficient pair} (or \emph{pair} for short) we shall mean a pair $(w, \mathtt{x})$, where $w \in \mathtt{A}^{\pm \ast}$ and $x \in \mathrm{Num}_{\mathfrak N}$, and by an {\ewfl} (or \emph{list} for short) we shall mean a doubly linked list of word-coefficient pairs. Any such list then represents a counting function over $F_n$. Explicitly, if $\mathcal L = ((w_1, x_1), \dots, (w_N, x_N))$, then the \emph{associated counting function} is
\[
\rho_{\mathcal L} = \sum_{i=1}^N \langle x_i \rangle \rho_{w_i}.
\]
As in the monoid case we say that a list  $\mathcal L = ((w_1, \mathtt{x}_1), \dots, (w_N, \mathtt{x}_N))$ has \emph{maximal depth}
\[
\ell := \sup\{|w_j| \mid \mathtt{x}_j \not \equiv 0\}
\]
and we say that $\mathcal L$ is \emph{minimal} if it is not equivalent to a list of smaller maximal depth. We also say that $\mathcal L$ is \emph{normalized} if the words $w_j$ are all distinct, ordered by length and ordered lexicographically within words of the same length (with respect to some fixed total order on $\mathtt{A}^\pm$), and if $\mathtt{x}_j \not \equiv 0$
for all $j \in \{1, \dots, N\}$.
\begin{rem}
As in the monoid case there is a procedure \textsc{NormalizeList} which replaces a given list $\mathcal L$ by an equivalent normalized list $\mathcal N$ of total size $|\mathcal N|_{\mathrm{tot}} \leq |\mathcal L|_{\mathrm{tot}}$ in time at most $O(T(|\mathcal L|_{\mathrm{tot}}))$.
\end{rem}
Similarly to the monoid case we have two kinds of basic equivalences between counting functions on $F_n$ (see \cite{HT1}), but these are now given by the slightly different formulas
\begin{equation}\label{ExtensionRelationsGroup}
\rho_w \sim \sum_{\mathtt{a} \in \mathtt{A}^\pm \setminus \{w_{\mathrm{in}}^{-1}\}} \rho_{\mathtt{a} w} \quad \text{and} \quad \rho_w \sim  \sum_{\mathtt{a} \in \mathtt{A}^\pm \setminus \{w_{\mathrm{fin}}^{-1}\}}  \rho_{w\mathtt{a}}.
\end{equation}
To take this into account, we will need to slightly modify our pruning and transfer moves.

We will in particular be interested in counting functions $f$ which are \emph{antisymmetric} (i.e. $f(g^{-1}) = -f(g)$), since these represent classes in bounded cohomology of $F_n$ (see \cite{Frigerio}). We then also call the corresponding lists antisymmetric.

As before, two lists are called \emph{equivalent} if the corresponding counting functions are equivalent, i.e.\ at bounded distance from each other, and this equivalence is denoted by $\sim$. We also say that two symmetric counting functions $f_1, f_2$, and by extension any two lists representing them, are \emph{cohomologous}, denoted $f_1 \equiv f_2$, provided $f_1-f_2$ is at bounded distance from a homomorphism and hence $f_1$ and $f_2$ define the same class in bounded cohomology. We want to solve the following two problems algorithmically:
\begin{problem}[Equivalence problem]\label{Prob1G} Given two lists $\mathcal L_1, \mathcal L_2$, decide whether $\mathcal L_1 \sim \mathcal L_2$.
\end{problem}
\begin{problem}[Cohomological problem]\label{Prob2G} Given two antisymmetric lists $\mathcal L_1, \mathcal L_2$, decide whether $\mathcal L_1 \equiv \mathcal L_2$.
\end{problem}
As in the monoid case, both problems can be reduced to the following problem - note that a list which is equivalent to an antisymmetric list is cohomologous to the empty list if and only if it is equivalent to  a list of maximal depth $\leq 1$.
\begin{problem}[Minimality problem]\label{MainProblemG} Given a list $\mathcal L$, find a minimal list $\mathcal L'$ which is equivalent to $\mathcal L$.
\end{problem}

\subsection{Brotherhoods and basic moves}
Let $T_n$ denote the right-Cayley tree of $F_n$ with respect to the extended generating set $\mathtt{A}^{\pm}$, considered as an oriented rooted tree with root $\varepsilon$ and edges oriented away from $\varepsilon$. We identify elements of $F_n$ with vertices of $T_n$. Within $T_n$ we can talk about fathers, brothers, brotherhoods and related brotherhoods as in the monoid case and we use the same notation. Note, however, that $|\{\varepsilon \ast\}| = 2n$ whereas $|\{w\ast\}| = 2n-1$ for all $w \neq \varepsilon$, i.e.\ the unique brotherhood of depth $1$ is different in size from all the other brotherhoods. Given a normalized list $\mathcal L$ and a brotherhood $B$ we define the weighted brotherhood $\mathcal L_B$ exactly as in the monoid case. We can then define a Procedure \textsc{DetachBrotherhood} with the same properties and runtime as in Lemma \ref{detaching-lemma} also in the group case.

\subsubsection{Pruning}

Let $\mathcal L$ be a normalized list of maximal depth $\ell \geq 1$. As in the monoid case, if $B = \{w\ast\}$ is a brotherhood and the weighted brotherhood $\mathcal L_B = ((w\mathtt{a}_1, \mathtt{x_1}), \dots, (w\mathtt{a_n}^{-1}, \mathtt{x}_{2n}))$ is non-empty and constant with $\mathtt{x}_1 \equiv \dots \equiv \mathtt{x}_{2n} \equiv \mathtt{x}_n \equiv \mathtt{x}$ for some $\mathtt{x} \in \mathtt{Num}_{\mathfrak N}$, then we can remove $\mathcal L_B$ from $\mathcal L$, append $\{\mathtt{w}, \mathtt{x}\}$ and normalize the resulting list. This is called a \emph{pruning move}, and a normalized list is called \emph{pruned} if it does not allow for any pruning moves.

We can now extend the procedure \textsc{PruneList} from the monoid case to the group case. Our new procedure carries out exactly the same steps as in the monoid case. This procedure will have the same properties listed in Lemma \ref{PCB-lemma} except that the estimate in Part \eqref{size-property} becomes
\[
|\mathcal L|_{\rm tot} \leq \frac{1}{h(n)}(|\mathcal N|_{\rm tot}-|\mathcal N'|_{\rm tot})
\]
where $h(n)$ is the size of a non-empty constant brotherhood of depth $\ell$, i.e. $h(n) = 2n-1$ if $\ell \geq 2$ and $h(n) = 2n$ if $\ell = 1$.

\subsubsection{Generic transfer-and-prune moves}

As in the monoid case we also have the notion of a transfer move in the group case. In fact we have two slightly different transfer moves, depending on whether the brotherhood, which is transferred, has depth $\geq 3$ (the \emph{generic} case) or depth $2$ (the \emph{special} case). In this subsection we discuss exclusively the generic case, leaving the special case to the next subsection. Thus let $u \in F_n \setminus \{\varepsilon\}$ and let $\mathcal L$ be a normalized list of depth $\ell \geq 3$, where $|u| = \ell - 2 \geq 1$. We set
\[
 \mathtt A^{\pm}_{\mathrm{in}} := \mathtt A^{\pm} \setminus \{u_{\mathrm{in}}^{-1}\} \quad \text{and} \quad \mathtt A^{\pm}_{\mathrm{fin}}  := \mathtt A^{\pm} \setminus \{u_{\mathrm{fin}}^{-1}\}.
\]
As in the monoid case we write $\mathcal L _{\{\ast u \ast\}}$ for the concatenation of the weighted brotherhoods $\mathcal L_{\{\mathtt{a} \mathtt{u} \ast\}}$ with $\mathtt{a} \in \mathtt A^{\pm}_{\mathrm{in}}$. Entries of $\mathcal L _{\{\ast u \ast\}}$  are then of the form
\begin{equation}\label{aa'coeff}
(\mathtt {a} \mathtt{u} \mathtt{a'}, \mathtt{x}_{\mathtt{a}\mathtt{a'}}) \quad \text{with}\quad a \in \mathtt A^{\pm}_{\mathrm{in}} \text{ and } \mathtt a' \in \mathtt A^{\pm}_{\mathrm{fin}}.
\end{equation}
We many thus define the encoded transfer matric $\mathtt T(\mathcal L, u)$ as the matrix of size $(2\ell-1) \times (2\ell -1)$ whose rows and columns are indexed by $\mathtt A^{\pm}_{\mathrm{in}}$ and  $\mathtt A^{\pm}_{\mathrm{fin}}$ respectively and whose entry $\mathtt t_{\mathtt{a} \mathtt{a'}}$ with index $(\mathtt{a}, \mathtt{a'}) \in \mathtt A^{\pm}_{\mathrm{in}}  \times \mathtt A^{\pm}_{\mathrm{fin}}$ is given by $\mathtt{x}_{\mathtt{a}\mathtt{a'}}$ if $\mathcal L _{\{\ast u \ast\}}$ contains an entry of the form \eqref{aa'coeff} and by $t_{\mathrm{a} \mathrm{a'}} := \varepsilon$ otherwise. If we fix $\mathtt{b} \in \mathtt A^{\pm}_{\mathrm{in}}$. then we have
\begin{eqnarray*}
\rho_{\mathcal L_{\{\ast u \ast\}}} &=& \sum_{\mathtt{a} \in \mathtt{A}^\pm_{\mathrm{in}}}  \sum_{\mathtt{a}' \in \mathtt{A}^\pm_{\mathrm{fin}}}  \mathtt{t}_{\mathtt a \mathtt a'}\rho_{\mathtt a u \mathtt a'} \; = \; \sum_{\mathtt{a} \in \mathtt{A}^\pm_{\mathrm{in}}}  \sum_{\mathtt{a}' \in \mathtt{A}^\pm_{\mathrm{fin}}} (\mathtt{t}_{\mathtt a \mathtt a'}-\mathtt{t}_{\mathtt b \mathtt a'})\rho_{\mathtt a u \mathtt a'}
 \rho_{\mathtt{a}_i\mathtt{u} \mathtt{a}_j} +   \sum_{\mathtt{a} \in \mathtt{A}^\pm_{\mathrm{in}}}  \sum_{\mathtt{a}' \in \mathtt{A}^\pm_{\mathrm{fin}}}\mathtt{t}_{\mathtt b \mathtt a'} \rho_{\mathtt a u \mathtt a'} \\
&\sim& \sum_{\mathtt{a} \in \mathtt{A}^\pm_{\mathrm{in}} \setminus\{\mathtt b\}}  \sum_{\mathtt{a}' \in \mathtt{A}^\pm_{\mathrm{fin}}} (\mathtt{t}_{\mathtt a \mathtt a'}-\mathtt{t}_{\mathtt b \mathtt a'})\rho_{\mathtt a u \mathtt a'}
+  \sum_{\mathtt{a}' \in \mathtt{A}^\pm_{\mathrm{fin}}}\mathtt{t}_{\mathtt b \mathtt a'} \rho_{u \mathtt a'}
\end{eqnarray*}
The list $\mathcal L$ is thus equivalent to any list $\mathcal L'$ which is obtained from $\mathcal L$ by deleting $\mathcal L_{\{\ast u \ast\}}$,  appending a pair of the form $(\mathtt{a}\mathtt{u} \mathtt{a}', \mathtt{y}_{\mathtt{a}\mathtt{a'}})$ with $\mathtt{y}_{\mathtt{a}\mathtt{a'}} \equiv \mathtt{t}_{\mathtt a \mathtt a'}-\mathtt{t}_{\mathtt b \mathtt a'}$ for all $\mathtt a \in \mathtt{A}^\pm_{\mathrm{in}} \setminus\{\mathtt b\}$ and $\mathtt{a}' \in \mathtt{A}^\pm_{\mathrm{fin}}$ with $\mathtt{t}_{\mathtt a \mathtt a'}\not \equiv \mathtt{t}_{\mathtt b \mathtt a'}$, appending a pair of the form $(\mathtt{u} \mathtt{a}', \mathtt{z}_{\mathtt{a'}})$ with $\mathtt{z}_{\mathtt{a'}} \equiv \mathtt{t}_{\mathtt b \mathtt a'}$ for every $\mathtt{a'} \in \mathtt{A}^\pm_{\mathrm{fin}}$with $ \mathtt{t}_{\mathtt b \mathtt a'} \neq \varepsilon$ and normalizing the resulting list. We say that any such list $\mathcal L'$ is obtained from $\mathcal L$ by a \emph{transfer move} with stem $u$ and special letter $\mathtt b$.

According to \cite[Theorem 4.2]{HT1}, Theorem \ref{StoppingCriterion} holds mutatis mutandis also in the group case, i.e.\ a list of maximal depth $\ell$ is minimal if it admits two related weighted brotherhoods one of which is empty and the other of which is non-constant. If $\mathcal L'$ is obtained from $\mathcal L$ by a transfer move with stem $u$ and special letter $\mathtt b$, then $\mathcal L'_{\mathtt{b}u*}$ is empty and hence $\mathcal L'$ is minimal unless all of the weighted brotherhoods of the form $\mathcal L'_{\mathtt{a}u*}$ with $\mathtt{a} \in  \mathtt{A}^\pm_{\mathrm{in}} \setminus\{\mathtt b\}$ are constant. As in the monoid case, this implies that $\mathcal L'$ is minimal unless $\mathtt{T}(\mathcal L, u)$ is a column-row sum.

Now assume that $\mathtt{T}(\mathcal L, u) = (\mathtt t_{\mathtt{a}\mathtt{a'}})$ is a column-row sum. This means that for every $\mathtt{a} \in \mathtt{A}^\pm_{\mathrm{in}}$ the
the number $\langle \mathtt{t_{aa'}} \ominus \mathtt{t}_{\mathtt{ba'}} \rangle$ is independent of $\mathtt{a'} \in  \mathtt{A}^\pm_{\mathrm{fin}}$. Consequently, if we choose elements $y_{\mathtt{a}}, z_{\mathtt{a'}} \in \Num_{\mathfrak N}$ such that
\begin{equation}\label{yzfinalGroup}
\mathtt y_{\mathtt{a}} \equiv  \mathtt{t_{aa'}} \ominus \mathtt{t}_{\mathtt{ba'}} \quad \text{and} \quad \mathtt z_{\mathtt{a}'} \equiv \mathtt{t}_{\mathtt b \mathtt a'} \quad (\mathtt{a} \in  \mathtt{A}^\pm_{\mathrm{in}}, \mathtt{a'} \in  \mathtt{A}^\pm_{\mathrm{fin}}),
\end{equation}
then from the above formula for $\rho_{\mathcal L_{\ast u \ast}}$ we obtain
\[
\rho_{\mathcal L_{\ast u \ast}} \sim \sum_{\mathtt{a} \in \mathtt{A}^\pm_{\mathrm{in}} \setminus\{\mathtt b\}}  \sum_{\mathtt{a}' \in \mathtt{A}^\pm_{\mathrm{fin}}} (\mathtt{t}_{\mathtt a \mathtt a'}-\mathtt{t}_{\mathtt b \mathtt a'})\rho_{\mathtt a u \mathtt a'}
+  \sum_{\mathtt{a}' \in \mathtt{A}^\pm_{\mathrm{fin}}}\mathtt{t}_{\mathtt b \mathtt a'} \rho_{u \mathtt a'} = \sum_{\mathtt{a} \in \mathtt{A}^\pm_{\mathrm{in}} \setminus\{\mathtt b\}}  \mathtt{y_a}\rho_{\mathtt a u}
+  \sum_{\mathtt{a}' \in \mathtt{A}^\pm_{\mathrm{fin}}}\mathtt{z}_{\mathtt{a'}} \rho_{u \mathtt a'}
\]
Thus if $\mathtt{T}(\mathcal L, u)$ is a column-row-sum and $\mathcal L'$ is obtained from $\mathcal L$ by deleting $\mathcal L_{\ast u \ast}$, appending pairs $(\mathtt{a}u, \mathtt{y_a})$ and $(u\mathtt{a'}, \mathtt{z_{a'}})$ subject to \eqref{yzfinalGroup} and normalizing, then $\mathcal L'$ is equivalent to $\mathcal L$. We say that $\mathcal L'$ is obtained from $\mathcal L$ by a \emph{transfer-and-prune move} with stem $u$ and special letter $\mathtt b$. Using this move we can now generalize the results of Section \ref{Sec:TransferAndPrune}:

\begin{lem}\label{LemmaGenericTransferAndPrune} There exists a procedure \textsc{TransferAndPrune} with the following properties:
\begin{enumerate}[(i)]
\item The input is a family $(\mathcal B_1, \dots, \mathcal B_n)$ of non-constant related brotherhoods of depth $\ell \geq 3$ such that $\mathcal B := \mathcal B_1 \cup \dots \cup \mathcal B_n$ is normalized.
\item The output is a Boolean variable $\textsc{minimal}$ and a list $\mathcal L$.
\item If $\textsc{minimal} = \texttt{true}$, then $\mathcal B$ is not equivalent to any list of depth $\leq \ell-1$ and $\mathcal L = \mathcal B$.
\item If $\textsc{minimal} = \texttt{false}$, then $\mathcal L$ is of constant depth $\ell-1$, equivalent to $\mathcal B$ and of total size
\[|\mathcal L|_{\rm tot} \quad \leq \quad  \frac{8}{9} \cdot |\mathcal B|_{\rm tot}.\]
\item The runtime of the procedure is $O(T(|\mathcal B|_{\rm tot}))$
\end{enumerate}
\end{lem}
There are two main differences to the results from Section \ref{Sec:TransferAndPrune}: Firstly, we allow $n=2$, whereas in the monoid case the condition $n \geq 3$ was needed to ensure the estimate in (iv). On the other hand, we have to assume here that $\ell \geq 3$; however, we will deal with the case $\ell = 2$ (which corresponds to an empty stem) separately.
\begin{proof} We follow as close as possible the algorithm of the same name described in Section \ref{Sec:TransferAndPrune}, in particular we will perform $7$ steps which correspond one-to-one to the $7$ steps in the monoid case.

In the first step we compute the stem and the transfer matrix just as in the monoid case. The transfer matrix has size $m \times m$, where $m := (2n-1)$ and since $n \geq 2$ we have $m \geq 3$. The transfer matrix is now indexed by $\mathtt A^{\pm}_{\mathrm{in}} \times \mathtt A^{\pm}_{\mathrm{fin}}$ rather than $\{1, \dots, m\} \times \{1, \dots, m\}$, but except for this change in indexing we can carry out Steps 2 and 3 as in the monoid case.

In Step 4, the sparseness condition has to be chosen relatively to the size $m$ of the matrix rather than relative to $n$. Thus we set \textsc{Sparse} to be true if and only if $m\geq 4$ (i.e. $n \geq 3$) and the transfer matrix $\mathtt{T}$ has less than $3m = 6n-3$ non-trivial entries or if $m=3$ (i.e. $n=2$) and $\mathtt{T}$ has less than $2m = 4n-2 = 6$ non-trivial entries.

The remaining Steps 5-7 are then carried out precisely as in the monoid case, except for the difference in indexing. For example, in Step \ref{non-sparse-case} (a) we append the pairs of the form
$(u\mathtt{a'}, \mathtt{t}_{\mathtt{b}\mathtt{a'}})$ for $\mathtt{a'}\in A^{\pm}_{\mathrm{fin}}$, where $\mathtt{b} \in \mathtt A^{\pm}_{\mathrm{in}}$ is chosen such that
\[
\sum_{\mathtt{a'} \in A^{\pm}_{\mathrm{fin}}} ||\mathtt{t}_{\mathtt{ba'}}|| = \min_{\mathtt{a}\in  \mathtt A^{\pm}_{\mathrm{in}}} \sum_{\mathtt{a'} \in A^{\pm}_{\mathrm{fin}}}  ||\mathtt{t}_{\mathtt{aa'}}||,
\]
and similarly for the other steps. It is clear that this change in indexing does not affect the runtime, nor the size of the output. The latter always satisfies (iv), since for all $n\geq 2$ we have $m\geq 3$ and hence the proof of Lemma \ref{SizeEstimationLemma} applies to the $(m \times m)$-matrix $\mathtt{T}$.
\end{proof}

\subsubsection{Special transfer-and-prune moves}\label{SecSpecial}
The transfer-and-prune move discussed in the previous section works for all $n \geq 2$ under the condition that the brotherhood, which is transferred, is of depth at least $3$. For brotherhoods of depth $\ell = 2$ there also exists a transfer and prune move, but this one is more complicated to describe. Nevertheless we can establish the following lemma:
\begin{lem} The statement of Lemma \ref{LemmaGenericTransferAndPrune} remains true also for $\ell = 2$ except that (iv) has to be replaced by
\begin{itemize}
\item[(iv$'$)] If $\textsc{minimal} = \texttt{false}$, then $\mathcal L$ is of constant depth $\ell-1 = 1$, equivalent to $\mathcal B$ and of total size
\[|\mathcal L|_{\rm tot} \quad \leq \quad \|\mathcal B\| + 2n,\]
and hence in particular $|\mathcal L|_{\rm tot} \leq  2n \cdot  |\mathcal B|_{\rm tot}$.
\end{itemize}
\end{lem}
In fact, this estimate is far from optimal, but it is easily established and sufficient for our purposes. For the proof of the lemma, let $\mathcal L$ be a normalized list of constant depth $2$. We fix a letter $\mathtt{b} \in \mathtt{A}^\pm$ (which we won't even bother to choose optimally) and set $\mathtt{A}^\pm_{\mathrm{in}} := \mathtt{A}^\pm \setminus \{\mathtt{b}^{-1}\}$. We want to define a move which clears out the sub-brotherhood of type $\{\mathtt{b}\ast\}$ in $\mathcal L$. We first define a \emph{transfer matrix} $\mathtt{T} = (\mathtt{t_{\mathtt{a}\mathtt{a'}}})$ over $\mathtt{A}^\pm \times \mathtt{A}^\pm$ by setting $\mathtt{t_{aa'}}$ to be the coefficient of $\mathtt{aa'}$ in $\mathcal L$ if such a word exists in $\mathcal L$ and setting $\mathtt{t_{aa'}} := \varepsilon$ otherwise. (In particular we then have $\mathtt{t}_{\mathtt{a}\mathtt{a}^{-1}} = \varepsilon$ for all $\mathtt{a} \in \mathtt{A}^\pm$.) Using the fact that for $a' \in \mathtt{A}^\pm_{\mathrm{in}} $ we have
\[
\langle \mathtt t_{\mathtt{ba'}}\rangle \rho_{\mathtt{ba'}} \sim \langle\mathtt t_{\mathtt{ba'}}\rangle\rho_{\mathtt{a'}} - \sum_{\mathtt a \neq (\mathtt{a'})^{-1}}  \langle\mathtt t_{\mathtt{ba'}}\rangle\rho_{\mathtt{aa'}},
\]
we obtain
\begin{eqnarray*}
\rho_{\mathcal L} &=& \sum_{\mathtt a \in \mathtt A^\pm} \sum_{\mathtt a' \neq \mathtt a^{-1}} \langle\mathtt{t_{aa'}}\rangle\rho_{\mathtt{aa'}}\\
&\sim& \sum_{a' \in \mathtt{A}^\pm_{\mathrm{in}}} \langle\mathtt t_{\mathtt{ba'}}\rangle\rho_{\mathtt{a'}} + \sum_{a \in \mathtt{A}^\pm\setminus\{b\}} \langle\mathtt{t}_{\mathtt{ab^{-1}}}\rangle\rho_{\mathtt{ab^{-1}}} + \sum_{\mathtt{a} \in \mathtt{A}^\pm \setminus\{\mathtt b\}}\sum_{\mathtt{a'}\in \mathtt{A}^\pm_{\mathrm{in}}\setminus\{\mathtt a^{-1}\}} \langle\mathtt{t}_{\mathtt{aa'}}\ominus \mathtt t_{\mathtt{ba'}}\rangle\rho_{\mathtt{aa'}}.
\end{eqnarray*}
We deduce that $\mathcal L$ is minimal unless for all $\mathtt{a} \in  \mathtt{A}^\pm \setminus\{\mathtt{b}\}$
\begin{equation}\label{SpecialPruningCondition}
\mathtt{t}_{\mathtt{ab}^{-1}} \equiv \mathtt{t}_{\mathtt{aa'}}\ominus \mathtt t_{\mathtt{ba'}} \quad \text{for all } \mathtt{a} \in  \mathtt{A}^\pm \setminus\{\mathtt{b}\} \text{ and } \mathtt{a'} \in  \mathtt{A}_{\mathrm{in}}^\pm \setminus\{\mathtt{a}^{-1}\}.
\end{equation}
On the other hand, if \eqref{SpecialPruningCondition} holds, then we can apply pruning moves to obtain
\[
\rho_{\mathcal L}  \sim \sum_{a' \in \mathtt{A}^\pm_{\mathrm{in}}} \langle \mathtt t_{\mathtt{ba'}}\rangle \rho_{\mathtt{a'}}+ \sum_{a \in \mathtt{A}^\pm\setminus\{b\}} \langle \mathtt{t}_{\mathtt{ab^{-1}}}\rangle \rho_{\mathtt{a}} = \langle \mathtt{t}_{\mathtt{bb}}\rangle \rho_{\mathtt{b}} + \langle \mathtt{t}_{\mathtt{b}^{-1}\mathtt{b}^{-1}}\rangle \rho_{\mathtt{b}^{-1}}\sum_{a \in \mathtt{A}^{\pm} \setminus\{b, b^{-1}\}} \langle \mathtt t_{\mathtt{ba}} \oplus \mathtt t_{\mathtt{ab}^{-1}}\rangle\rho_{\mathtt{a}}.
\]
This shows that the following procedure works correctly:

\noindent\hrulefill
\negthickspace
\begin{center}
Procedure \textsc{SpecialTransferAndPrune}
\end{center}
\noindent \textsc{Input}: An family of non-constant related brotherhoods $(\mathcal B_1, \dots, \mathcal B_n)$ of depth $\ell = 2$\\
\phantom{Input:*} such that $\mathcal B := \mathcal B_1 \cup \dots \cup \mathcal B_n$ is normalized.

\noindent \textsc{Output}:  A Boolean variable \textsc{minimal} and a list $\mathcal L$.

\begin{enumerate}
\item \label{Smatrix-stem-construction}
Compute the transfer matrix $\mathtt{T} = (\mathtt{t_{aa'}})$ and choose $\mathtt{b} := \mathtt{a}_1$.

\item \label{Sbalance-check}
For each $\mathtt{a} \in  \mathtt{A}^\pm \setminus\{\mathtt{b}\}$ do\\
for each $\mathtt{a'} \in  \mathtt{A}_{\mathrm{in}}^\pm \setminus\{\mathtt{a}^{-1}\}$
do the following:
\begin{enumerate}[(a)]
\item If $\mathtt{t}_{\mathtt{ab}^{-1}} \not \equiv \mathtt{t}_{\mathtt{aa'}}\ominus \mathtt t_{\mathtt{ba'}} $, then \\ return $\textsc{Minimal} := \texttt{true}$ and $\mathcal L := \mathcal B_1 \cup \dots \cup \mathcal B_n$  and terminate the procedure\\
\end{enumerate}
\item Set $\mathcal L := ((\mathtt b, \mathtt{t}_{\mathtt{bb}}), (\mathtt{b}^{-1}, \mathtt{t}_{\mathtt{b^{-1}b^{-1}}}))$ and for $\mathtt{a} \in  \mathtt{A}^{\pm} \setminus\{b, b^{-1}\}$ do the following:
\begin{enumerate}[(a)]
\item Append $(\mathtt a, \mathtt t_{\mathtt{ba}} \oplus \mathtt t_{\mathtt{ab}^{-1}})$ to $\mathcal L$.
\end{enumerate}
\item Return  $\textsc{Minimal} := \texttt{false}$ and $\mathcal L$.
\end{enumerate}
\negthickspace
\noindent \hrulefill

It is easy to see that the runtime of the procedure is given by $O(T(|\mathcal B|_{\mathrm{tot}}))$. As for (iv$'$), in the non-minimal case we have $\|\mathcal L\| \leq \|\mathtt{T}\| = \|\mathcal B\|$ since each non-trivial coefficient of $\mathtt{T}$ is copied at most once into $\mathcal L$, and clearly $|\mathcal L| \leq |\mathtt{A}^\pm| = 2n$, hence (iv$'$) holds.

\subsection{The algorithm in the group case}
\subsubsection{The main processing step}
At this point we have extended the procedures \textsc{NormalizeList}, \textsc{DetachBrotherhood}, \textsc{PruneList} and \textsc{TransferAndPrune} from the monoid case to the group case. The latter works only for lists of depth $\ell \geq 3$, but the case $\ell = 2$ can be dealt with by the additional procedure \textsc{SpecialTransferAndPrune} from Section \ref{SecSpecial}. Using these procedures we can now define a procedure \textsc{MainProcessingStep} almost literally as in the monoid case (see p.\pageref{MPS-procedure}), except that in Step 4(b) we replace the procedure \textsc{TransferAndPrune} by the procedure \textsc{SpecialTransferAndPrune} if $\ell = 2$. The same analysis as in the monoid case then shows:

\begin{lem}
\label{GMPS-lemma} For every non-abelian free group $F_n$ with $n\geq 2$, there exists an algorithm \textsc{MainProcessingStep} with the following properties:
\begin{enumerate}[(i)]
\item The input is a normalized list $\mathcal L$ of constant depth $\ell \geq 2$.
\item The output is a boolean parameter $\textsc{minimal}$ and a normalized list $\mathcal L'$ equivalent to $\mathcal L$.
\item It $\textsc{minimal} = \mathtt{true}$, then $\mathcal L'$ is minimal.
\item If $\textsc{minimal} = \mathtt{false}$, then $\mathcal L'$ is of constants depth $\ell - 1$ and
\begin{equation}
|\mathcal L'|_{\mathrm{tot}} \leq \left\{\begin{array}{rl} \frac89|\mathcal L|_{\rm tot}, & \text{if }\ell \geq 3,\\\
2n|\mathcal L|_{\rm tot}, & \text{if }\ell =2.
\end{array}\right.
\end{equation}
\item The runtime of the algorithm is $O(T(|\mathcal L|))$.
\end{enumerate}
\end{lem}
Here the difference in (iv) (compared to Lemma \ref{MPS-lemma}) comes from the use of the procedure \textsc{SpecialTransferAndPrune}.

\subsubsection{The final algorithm}
We are now ready to establish the main result of the present article; as before we set $T(N) := N$ if $\mathfrak N = \Z$ and $T(N) := N \log N$ if $\mathfrak N = \Q$.
\begin{thm}
\label{ComplexityGroup} For every $n\geq 2$ and $\mathfrak N \in \{\Z, \Q\}$ there exists an algorithm \textsc{FindMinimalList} which takes as input an {\ewfl} $\mathcal L$ over $F_n$ with coefficients in $\mathfrak N$, terminates within time  $O(T(|\mathcal L|_{\rm tot}))$ and gives as output a minimal {\ewfl} $\mathcal M$ equivalent to $\mathcal L$.
\end{thm}
This implies the following more precise version of of Theorem \ref{Main1} from the introduction:
\begin{cor}\label{CorMainThm} For every $n \geq 2$ and $\mathfrak N \in \{\Z, \Q\}$ there exist algorithms of time complexity $O(T(N))$ (where $N$ denotes the size of the input) to decide whether two counting functions (respectively counting quasimorphisms) over $F_n$ with coefficients in $\mathfrak N$ (encoded as \ewfl s) are equivalent (respectively cohomologous).\qed
\end{cor}
The proof of Theorem \ref{ComplexityGroup} is analogous to the proof of Theorem \ref{ComplexityMonoid} in the monoid case. In fact, if we use the group versions of the procedures \textsc{NormalizeList}, \textsc{PruneList} and \textsc{MainProcessingStep} instead of the monoid versions, then we can define the algorithm \textsc{FindMinimalList} literally as in the monoid case. The proof for correctness of this algorithm is then as in the monoid case.

As for the runtime analysis of the algorithm in the group case, there is only one difference compared to the monoid case, which is caused by the difference between Lemma \ref{MPS-lemma} and Lemma \ref{GMPS-lemma}. Namely, in Step \ref{RunningStep} we apply the procedure \textsc{MainProcessingStep} to generate lists $\mathcal M_d, \dots, \mathcal M_t$; here either $t = 1$ or $t \geq 2$ and a break happens before $\mathcal M_{t-1}$ was computed. As long as $i \in \{t,\dots, d\}$ satisfies $i \geq 2$ we have
\begin{equation*}
\label{RecursiveInequality}
|\mathcal M_{i}|_{\rm tot} \leq \frac89|\mathcal  M_{i+1}|_{\rm tot} + |\mathcal N_i|_{\rm tot},
\end{equation*}
just as in the monoid case, but in view of the difference between Lemma \ref{MPS-lemma} and Lemma \ref{GMPS-lemma} we only obtain the weaker estimate
\[
|\mathcal M_1|_{\rm tot} \leq 2n |\mathcal  M_{2}|_{\rm tot} +  |\mathcal N_1|_{\rm tot}.
\]
Of course, this difference only occurs if $t=1$, i.e.\ if no break occurs in Step \ref{RunningStep}. Assume this from now on. In this case we have, as in the monoid case, the estimates
\[
|\mathcal M_2|_{\rm tot} \leq \sum_{i=2}^{d}|\mathcal M_i|_{\rm tot} \leq 9\sum_{i=2}^d|\mathcal N_{i}|_{\rm tot},
\]
and hence
\[
\sum_{i=1}^d |\mathcal M_i|_{\rm tot} \leq  |\mathcal M_1|_{\rm tot} + 9\sum_{i=2}^d|\mathcal N_{i}|_{\rm tot} \leq  2n |\mathcal  M_{2}|_{\rm tot} + 9\sum_{i=1}^d|\mathcal N_{i}|_{\rm tot} \leq (18n+9) \cdot \sum_{i=1}^d|\mathcal N_{i}|_{\rm tot},
\]
i.e. for all $i \in \{1,\dots, d\}$ we obtain
\[
|\mathcal M_i| \leq \sum_{i=1}^d |\mathcal M_i|_{\rm tot} \leq (18n+9) \left(|\mathcal N|_{\rm tot} - |\mathcal N_0|_{\rm tot}\right).
\]
This inequality can now be used to replace Inequality \eqref{MainInequality} for the remainder of the runtime analysis. The rest of the proof of Theorem \ref{ComplexityGroup} is then identical with the proof of Theorem \ref{ComplexityMonoid}, except for the slightly worse constant $(18n+9)$ instead of $9$. Of course, for this argument to be valid it is important for $n$ to be so small as to be considered as a constant.

\newpage
\appendix

\section{Encoding of arithmetic operations}\label{AppArithmetic}

In this appendix we discuss the encodings of arithmetic for integer and rational numbers which are used in the body of the text. For each of two cases $\mathfrak N \in \{\Z, \Q\}$ we will choose an alphabet $\Sigma_{\mathfrak N}$, an encoding subset $ \Num_{\mathfrak N} \subset \Sigma^*$ of the set $\Sigma^*$ of words over $\Sigma$ and a surjective map
\[
\Num_{\mathfrak N} \to \mathfrak{N}, \quad \mathtt{x} \mapsto \langle \mathtt{x} \rangle.
\]
Given $\mathtt{x}, \mathtt{y} \in \Num_{\mathfrak N}$ we will write $\mathtt{x} \equiv \mathtt{y}$ provided $\langle \mathtt{x} \rangle = \langle \mathtt{y} \rangle$.

\subsection{Encoding integer arithmetic}
To encode the semiring $\mathbb N_0$ of non-negative integers we choose an auxiliary alphabet $\Sigma_{\mathbb N_0}:=\{\mathtt{0},\mathtt{1}\}$. We then define $\Num_{\mathbb N_0}$ as the union of the singleton $\{\mathtt{0}\}$ and the set of all finite words over $\Sigma_{\mathbb N_0}$ which start with $\mathtt{1}$. We then obtain a bijective encoding
\[
\Num_{\mathbb N_0} \to \mathbb N_0, \quad \mathtt{x} \mapsto \langle \mathtt{x} \rangle
\]
by interpreting each word in $\Num_{\mathbb N_0}$ as a binary expansion of a natural number, so that e.g.\  $\mathtt{111}$ represents $7$. To encode the ring $\Z$ of integers we use the alphabet $\Sigma_{\mathbb Z}:=\{\mathtt{0},\mathtt{1},\mathtt{+},\mathtt{-}\}$ and define
\[
\Num_{\mathfrak \Z} := \{\varepsilon\} \cup \{\mathtt{+}\mathtt{x} \mid \mathtt{x} \in \Num_{\mathbb N_0}\} \cup \{\mathtt{-}\mathtt{x} \mid \mathtt{x} \in \Num_{\mathbb N_0}\},
\]
where $\varepsilon$ denotes the empty word. We then have an encoding
\[
\Num_{\mathbb Z} \to \mathbb Z, \quad \mathtt{x} \mapsto \langle \mathtt{x} \rangle
\]
as follows: The empty word is interpreted as $0$, and for non-empty words we interpret the first letter as the sign and the rest of the word as the binary expansion of the absolute value, e.g.  $\mathtt{-111}$ represents $-7$. This encoding is almost injective except that $\varepsilon \equiv \mathtt{+0} \equiv \mathtt{-0}$. This non-uniqueness will sometimes be convenient for us. If $\mathtt{x} \in \Num_{\Z} \setminus \{\varepsilon\}$, then the amount of memory used to store $\mathtt{x}$, i.e. its word length in the alphabet $\Sigma_\Z$ is given by
\[
|\mathtt{x}|_{\Sigma_\Z} := \lceil\log_2(|\langle \mathtt{x} \rangle| + 1)\rceil + 1,
\]
whereas $|\varepsilon|_{\Sigma_\Z} = 0$. Given $\mathtt{x}\in \Sigma_\Z$ we also write $\|x\| := |\mathtt{x}|_{\Sigma_\Z}$ and refer to $\|\mathtt{x}\|$ as the \emph{size} of $\mathtt{x}$. The following is a standard elementary exercise in the theory of computing.
\begin{lem}
\label{Z-size-lemma}
Let $\mathtt{x}_1,\mathtt{x}_2 \in \Num_\Z$. Then there exist elements $\mathtt{x}_3 =: \mathtt{x}_1 \oplus \mathtt{x}_2 \in \Num_\Z$ and $\mathtt{x}_4  =: \mathtt{x}_1 \ominus \mathtt{x}_2 \in \Num_\Z$ of size $\|\mathtt{x}_3\|, \|\mathtt{x}_4\|\leq \max\{\|\mathtt{x}_1\|,\|\mathtt{x}_2\|\}$ which can be computed in time $O(\max\{\|\mathtt{x}_1\|,\|\mathtt{x}_2\|\}+1)$ and satisfy $\langle \mathtt{x}_3 \rangle = \langle \mathtt{x}_1 \rangle +  \langle \mathtt{x}_2 \rangle$ and  $\langle \mathtt{x}_4 \rangle = \langle \mathtt{x}_1 \rangle -  \langle \mathtt{x}_2 \rangle$.  \qed
\end{lem}
\begin{rem} Besides adding and subtracting integers, we also need to be able to decide whether to given words in $\Num_\Z$ represent the same integer. It turns out that equality can be checked in linear time. Indeed, for words of length $\geq 3$ one just has to compare the words over $\Sigma_\Z$. For shorter words the only relation one has to account is
the fact that $\mathtt{+0} \equiv \mathtt{-0} \equiv \varepsilon$, but this can be checked in bounded time as well.
\end{rem}
While addition, subtraction and comparison of large integers can be performed in linear time (compared to the size of the input), this is no longer the case for multiplication. In fact, the construction of efficient multiplication algorithms for large integers is an important problem in the theory of computation. In practice, multiplication is usually realized by the Toom--Cook algorithm (see \cite{Toom, Cook}), which for medium-sized inputs is one of the fastest known algorithm. Given $\mathtt{x}_1, \mathtt{x}_2 \in \Num_\Z$ it computes an expression $\mathtt{x}_1 \odot \mathtt{x}_2 \in \Num_\Z$ representing $\langle \mathtt{x}_1\rangle \cdot \langle \mathtt{x}_2 \rangle$ in time $O(T_1(\|\mathtt{x}_1\|+\|\mathtt{x}_2\|))$, where
$T_1(N)=N^{\theta}$ with $\theta=\log 5/\log 3\approx 1.465$.

However, for theoretical purposes (or very large inputs), there are multiplication algorithms which are even faster: Sch\"onhage and Strassen gave an algorithm which given $\mathtt{x}_1, \mathtt{x}_2 \in \Num_\Z$ computes an expression $\mathtt{x}_1 \odot \mathtt{x}_2 \in \Num_\Z$ representing $\langle \mathtt{x}_1\rangle \cdot \langle \mathtt{x}_2 \rangle$ in time $O(T_2(\|\mathtt{x}_1\|+\|\mathtt{x}_2\|))$, where  $T_2(N) := N \log(N) \log\log(N)$ and conjectured that the optimal time complexity of such an algorithm should be $O(T_3(\|\mathtt{x}_1\|+\|\mathtt{x}_2\|))$, where $T_3(N) := N \log(N)$. An algorithm with this time complexity was provided very recently by Harvey and van der Hoeven (see \cite{Harvey-vanderHoeven}), but it is still unknown whether its time complexity is optimal. Moreover, no superlinear lower bound is currently known.
\begin{convention}\label{ConvT} From now on we fix a function $T$ with the following properties:
\begin{enumerate}[(T1)]\label{TProperties}
\item $T$ is superadditive, i.e. $T(N_1+N_2) \geq T(N_1) + T(N_2)$.
\item $T$ is asymptotically at least linear, i.e. $N= O(T(N))$.
\item For each $C>1$ one has $T(N) = O(T(CN))$.
\item Multiplication of integers of size $N_1$ and $N_2$ can be performed in time $O(T(N_1+N_2))$.
\end{enumerate}
\end{convention}
It is easy to check that the functions $T_1$, $T_2$ and $T_3$ above all satisfy conditions (T1)--(T3); they also satisfy (T4) by the work of Toom--Cook, Sch\"onhage--Strassen and Harvey--van der Hoeven respectively.

\subsection{A complexity convexity lemma}

In the previous subsection we have seen that the time complexity of integer multiplication is governed by a function $T$ which is both superadditive and asymptotically at least linear. These properties of $T$ have the following crucial consequence:
\begin{lem}\label{CuttingLemma}
Let $g$ be a real-valued function and let $T$ be a real-valued function satisfying Properties (T1) and (T2) of Convention \ref{ConvT}.
If $g(x) = O(T(x))$, then
\[
g(x_1) + \dots + g(x_k) = O(T(x_1 + \dots + x_k)),
\]
where  the implied constants are independent of $k$.
\end{lem}
More explicitly, the conclusion of the lemma says that
\[
\limsup \left\{\frac{g(x_1)+ \dots + g(x_k)}{T(x_1 + \dots + x_k)} \mid k \in \mathbb N, x_1, \dots, x_k \in \mathbb N\right\}< \infty.
\]
\begin{proof} Let $x := x_1 + \dots + x_k$. Since $x_j \in \mathbb N$ we have $x_j \geq 1$ for all $j=1, \dots, k$ and hence $k \leq x$. Since $g(x) = O(T(x))$ there exist constants $C_0, C_1, C_2 \in \mathbb N$ such that
\[
g(x) \leq \left\{\begin{array}{ll}
C_1 T(x), & \text{if }x > C_0,\\
C_2, & \text{if }x \leq C_0.
 \end{array}\right.
\]
Up to reordering the $x_j$ we may assume that $x_1, \dots, x_r > C_0$ and $x_{r+1}, \dots, x_k \leq C_0$. Then using superadditivity of $T$ we obtain
\begin{eqnarray*}
g(x_1) + \dots + g(x_k) &\leq& C_1T(x_1)+ \dots +C_1T(x_r) + C_2 + \dots + C_2\\
&=& C_1(T(x_1) + \dots + T(x_r)) + C_2(k-r)\\
&\leq& C_1 T(x_1+ \dots + x_r) + C_2 k\\
&\leq& C_1T(x) + C_2 x.
\end{eqnarray*}
Since $T$ is asymptotically at least linear, the lemma follows.
\end{proof}
The relevance of this property for our runtime analysis is explained in Remark \ref{SplittingLists}.

\subsection{Encoding arithmetic of rational numbers}

We now encode the field $\Q$ of rational numbers. We will store such numbers as mixed fractions, and it will be crucial for many of our algorithms to allow \emph{unreduced mixed fractions} to appear in the computations. In particular, this is necessary in order to implement addition of rational numbers efficiently (since reducing fractions takes even more time than multiplication). Working with unreduced mixed fractions will have the side-effect that infinitely many different words will represent the same rational number, but this will not cause any problems.\\

To define our encoding we set $\Sigma_{\mathbb Q} := \{\mathtt{0},\mathtt{1},\mathtt{+},\mathtt{-},\mathtt{\slash}\}$ and define
\[
\Num_{\mathfrak \Q} := \{\varepsilon\} \cup \{\mathtt{s} \mathtt{K}/\mathtt{m}/\mathtt{n}\mid \mathtt{s} \in \{\mathtt{+}, \mathtt{-}\},\; \mathtt{K,m,n} \in \Num_{\mathbb N_0}, \;\langle \mathtt{m} \rangle < \langle \mathtt{n} \rangle \}.
\]
We then define an encoding
\[
\Num_{\mathbb Q} \to \mathbb Q, \quad \mathtt{x} \mapsto \langle \mathtt{x} \rangle
\]
as follows: The emptyword is interpreted as $0$, and if $K,m,n \in \Num_{\mathbb N_0}$, then we set
\[
\langle \mathtt{s} \mathtt{K}/\mathtt{m}/\mathtt{n} \rangle := \langle \mathtt{s}\mathtt{K} \rangle + \frac{\langle \mathtt{m}\rangle}{\langle \mathtt{n} \rangle}.
\]
Since we do not require $\mathtt{m}$ and $\mathtt{n}$ to be relatively prime, the map $\mathtt{x} \to \mathtt{\langle x \rangle}$ is not injective. For example, $\mathtt{-11/10/101}$ and $\mathtt{-11/100/1010}$ both represent $-3\frac25 = -3 \frac 4 {10}$. If $ \mathtt{s} \mathtt{K}/\mathtt{m}/\mathtt{n}  \in  \Num_{\Q} \setminus \{\varepsilon\}$, then the amount of memory needed to store $ \mathtt{s} \mathtt{K}/\mathtt{m}/\mathtt{n} $ is given by
\[
| \mathtt{s} \mathtt{K}/\mathtt{m}/\mathtt{n}|_{\Sigma_\Q}=\lceil\log_2(|\langle \mathtt{K} \rangle|+1)\rceil+\lceil\log_2(|\langle \mathtt{m} \rangle|+1)\rceil+\lceil\log_2(|\langle \mathtt{n} \rangle|+1)\rceil+3,
\]
whereas $|\varepsilon|_{\Sigma_\Q} = 0$. However, it will be computationally convenient to work with a slightly different notion of size: We define the \emph{size} of an element of $\Num_{\mathfrak \Q}$ by $\|\varepsilon\| := 0$ and
\begin{equation}\label{SizeQ}
\| \mathtt{s} \mathtt{K}/\mathtt{m}/\mathtt{n}\| := \lceil\log_2(|\langle \mathtt{K} \rangle|+1)\rceil+2\lceil\log_2(|\langle \mathtt{n} \rangle|+1)\rceil+3,
\end{equation}
Since $0\leq m < n$ we then have
\begin{equation}
\|x\|<|x|_{\Sigma_\Q}<2\|x\|,
\label{StandartNonStandart}
\end{equation}
hence we do not lose anything by working with the size $\|\cdot\|$. The main advantage of our notion of size is that it behaves better with respect to arithmetic operations, as the following lemma shows.
\begin{lem}
\label{Q-size-lemma}
Let $T$ be a function as in Convention \ref{ConvT}. Then for all $\mathtt{x}_1,\mathtt{x}_2 \in \Num_\Q$ there exist elements $\mathtt{x}_3 =: \mathtt{x}_1 \oplus \mathtt{x}_2 \in \Num_\Q$ and $\mathtt{x}_4  =: \mathtt{x}_1 \ominus \mathtt{x}_2 \in \Num_\Q$ of size $\|\mathtt{x}_3\|, \|\mathtt{x}_4\| \leq \|\mathtt{x}_1\| + \|\mathtt{x}_2\|$ which can be computed in time $O(T(\|\mathtt{x}_1\|+\|\mathtt{x}_2\|))$ and satisfy $\langle \mathtt{x}_3 \rangle = \langle \mathtt{x}_1 \rangle +  \langle \mathtt{x}_2 \rangle$ and  $\langle \mathtt{x}_4 \rangle = \langle \mathtt{x}_1 \rangle -  \langle \mathtt{x}_2 \rangle$.
\end{lem}
\begin{proof} The case of subtraction follows immediately from the case of addition, so we will focus on the latter. The cases where either $\mathtt{x}_1$ or $\mathtt{x}_2$ are empty are obvious.

Thus assume $\mathtt{x}_1 = \mathtt{s_1K_1/m_1/n_1}$ and $\mathtt{x}_2 = \mathtt{s_2K_2/m_2/n_2}$ and let $K_1, m_1, n_1, K_2, m_2, n_2$ denote the respective interpretations of $\mathtt{K_1, m_1, n_1, K_2, m_2, n_2}$.

First, consider the case when both signs $\mathtt{s_1}$ and $\mathtt{s_2}$ are positive. The sum of the fractional parts of $\mathtt{x}_1$ and $\mathtt{x}_2$ is
\[
\frac{m_1}{n_1}+\frac{m_2}{n_2}=\frac{p}{q}, \quad \text{where }{p:= m_1n_2+m_2n_1}\text{ and }{q := n_1n_2}.
\]
We thus define
\[
\mathtt{p}:= (\mathtt{m}_1 \otimes \mathtt{n}_2) \oplus (\mathtt{m}_2 \otimes \mathtt{n}_1)\text{ and }\mathtt{q} := \mathtt{n}_1 \otimes \mathtt{n}_2.
\]
It is clear from the formulas that $\mathtt{p}$ and $\mathtt{q}$ both have size $O(||\mathtt{x}_1||+||\mathtt{x}_2||)$ and in view of Properties (T1) and (T3) of $T$ we can compute them in time $O(T(||\mathtt{x}_1||+||\mathtt{x}_2||))$ using $3$ multiplication and $2$ addition routines.

Once $\mathtt{p}$ and $\mathtt{q}$ have been computed, we can compute the integer part of the fraction $p/q$ in linear time  $O(\|\mathtt{x}_1\|+\|\mathtt{x}_2\|)$. Indeed, since $0\leq m_1<n_1$ and $0\leq m_2<n_2$, we have $m_1n_2<n_1n_2$ and $m_2n_1<n_1n_2$, hence $p<2q$. Then $K'=\lfloor p/q \rfloor$  can be equal to $0$ or $1$, so to find $K'$  we need only to check the condition $p<q$. If it is true, then $K'=0$, otherwise $K'=1$, and this can be checked in time $O(\|\mathtt{x}_1\|+\|\mathtt{x}_2\|)$. If $K' = 1$ we can also compute  $\mathtt{p}'=\mathtt{p}\ominus \mathtt{q}$ in linear time. Note that in this case $p/q=K'\frac{p'}q$, where $p' := \langle \mathtt{p}' \rangle$.

Finally, define $\mathtt{K} := \mathtt{K_1} \oplus \mathtt{K}_2$ if $K' = 0$ and $\mathtt{K} := \mathtt{K}_1\oplus \mathtt{K}_2 \oplus (\mathtt{+1})$ if $K' = 1$. Again, this can be computed in linear time. If $K' = 0$ then we define $\mathtt{x}_3 := \mathtt{+K}/\mathtt{p}/{\mathtt{q}}$. If $K' = 1$, then we define $\mathtt{x}_3 := \mathtt{+K}/\mathtt{p'}/{\mathtt{q}}$. In either case, $\mathtt{x}_3$ represents $ \langle \mathtt{x}_1 \rangle +  \langle \mathtt{x}_2 \rangle$ and has been computed in time  $O(T(\|\mathtt{x}_1\| + \|\mathtt{x}_2\|))$.

It remains to prove the inequality $||\mathtt{x}_3||\leq ||\mathtt{x}_1||+||\mathtt{x}_2||$.  For any $n_1,n_2\geq 1$ we have
\[
\lceil\log_2(n_1n_2+1)\rceil \leq \log_2((n_1+1)(n_2+1))+1 \leq  \log_2(n_1+1)+\log_2(n_2+1)+1.
\]
Hence the denominator $n_1n_2$ can be stored as a word of length
\begin{equation}
\lceil\log_2(n_1+1)\rceil+\lceil\log_2(n_2+1)\rceil+1.
\label{FractionalPartSize}
\end{equation}
Since the nominator of the fractional part (which is $p$ or $p'$) is strictly smaller than the denominator $n_1n_2$, it can be stored as a word of smaller length.

Similarly, for any $K_1,K_2 \geq 1$ we have
\[
\lceil\log_2(K_1+K_2+2)\rceil \leq \log_2((K_1+1)(K_2+1))+1 \leq \log_2(K_1+1)+\log_2(K_2+1)+1.
\]
In the cases $0\leq K_1, K_2 \leq 1$ this inequality is also true and is easily verified. This assures that $\mathtt{K}$ can be stored using $\lceil\log_2(K_1+1)\rceil+\lceil\log_2(K_2+1)\rceil+1$ symbols. Adding to this $1$ symbol for the sign, $2$ symbols for the separators and two values of \eqref{FractionalPartSize} which we use to save the fractional part, we see that $\mathtt{x}_3$ less than $N$ symbols altogether, where
\[
N=\lceil\log_2(K_1+1)\rceil+\lceil\log_2(K_2+1)\rceil+2(\lceil\log_2(n_1+1)\rceil+\lceil\log_2(n_2+1)\rceil)+6 = ||\mathtt{x}_1||+||\mathtt{x}_2||.
\]
This finishes the proof in the case $\mathtt{s}_1=\mathtt{s}_2=\mathtt{+}$.

\smallskip

If the signs $s_1$ and $s_2$ are both negative, we make the same computations, but define instead $\mathtt{x}_3 := \mathtt{-K}/\mathtt{p}/{\mathtt{q}}$ (or $\mathtt{x}_3 := \mathtt{-K}/\mathtt{p'}/{\mathtt{q}}$ if $K'=1$). Clearly this does not influence the result.

If the signs $s_1$ and $s_2$ are different then the computations are similar, except that some sums are replaced by differences and we need a few more comparisons to produce the integer part and the sign. The time complexity is still $O(T(\|\mathtt{x}_1\| + \|\mathtt{x}_2\|))$, and the size estimations are exactly the same since the denominator of the fractional part is again equal to $n_1n_2$ and the absolute value of the integral part does not exceed $|K_1|+|K_2|+1$.
\end{proof}
Concerning comparison of rational numbers we have the following statement:
\begin{lem}
Let $\mathtt{x}_1,\mathtt{x}_2 \in \Num_\Q$. Then we can decide whether $\mathtt{x}_1 \equiv \mathtt{x}_2$ or not in time $O(T(||\mathtt{x}_1||+||\mathtt{x}_2||))$.
\label{ComparisonLemma}
\end{lem}
\begin{proof} With the same notation as in the proof of Lemma~\ref{Q-size-lemma}, the lemma follows from the straight-forward check of one of the equalities $K_1n_1+m_1=K_2n_2+m_2$ or $K_1n_1+m_1=-(K_2n_2+m_2)$.
\end{proof}

\bigskip

\noindent\textbf{Authors' addresses:}\\
\noindent Tobias Hartnick \\
\noindent \textsc{Institut für Algebra und Geometrie, KIT, \\
Englerstr. 2, 76131 Karlsruhe, Germany}\\
\texttt{tobias.hartnick@kit.de};

\bigskip

\noindent Alexey Talambutsa \\
\noindent \textsc{Steklov Mathematical Institute of Russian Academy of Sciences, \\ Gubkina Str. 8, 119991, Moscow, Russia} \\
\texttt{altal@mi-ras.ru}



\begin{thebibliography}{9}

\bibitem{Brooks}
R. Brooks, \emph{Some remarks on bounded cohomology}, In: Riemann Surfaces and Related Topics: Proceedings of the 1978 Stony Brook Conference, Annals of Mathematics Studies, Princeton University Press, 1980, 53 -- 63.


\bibitem{scl}
D.~Calegari.
\newblock \emph{scl}, volume~20 of {\em MSJ Memoirs}.
\newblock Mathematical Society of Japan, Tokyo, 2009.


\bibitem{Cook}
S.~Cook, \emph{On the Minimum Computation Time of Functions} Thesis, Harvard University, 1966.

\bibitem{CormenEtAl}
T. H. Cormen, C. E. Leiserson, R. L. Rivest, and C. Stein, \emph{Introduction to Algorithms, Second Edition}. MIT Press and McGraw-Hill, 2001. 

\bibitem{Frigerio}
R. Frigerio, \emph{Bounded Cohomology of Discrete Groups}. American Mathematical Society, 2017. 

\bibitem{Grigorchuk} R. I. Grigorchuk, \emph{Some results on bounded cohomology}, In: Combinatorial and geometric group theory (Edinburgh, 1993),  London Math. Soc.
    Lecture Note Ser., 204, Cambridge Univ. Press, Cambridge, 1995. 111--163.

\bibitem{Hartnick-Schweitzer}
T. Hartnick, P. Schweitzer, \emph{On quasi-outomorphism groups of free groups and their transitivity properties}, Journal of Algebra \textbf{450} (2016), 242--281.

\bibitem{HSi}
T. Hartnick, A. Sisto, \emph{Bounded cohomology and virtually free hyperbolically embedded subgroups}, Groups, Geometry and Dynamics, \textbf{13:2} (2019), 677--694.

\bibitem{HT1}
T. Hartnick, A. Talambutsa, \emph{Relations between counting functions on free groups and free monoids}, Groups, Geometry and Dynamics, \textbf{12:4} (2018), 1485--1521.

\bibitem{Hase}
A. Hase, \emph{The ${\rm Out}(F_n)$-action on $H^2_b(F_n)$}, Preprint, \texttt{arXiv:1805.00366}

\bibitem{Harvey-vanderHoeven}
D. Harvey, J. van der Hoeven, \emph{Integer multiplication in time $O(n\log n)$}, Annals of Mathematics \textbf{193} (2021), 563--617.

\bibitem{HopcroftEtAl}
J. E. Hopcroft, R. Motwani, J. D. Ullman, \emph{Introduction to automata theory, languages, and computation}, 3rd ed., Pearson Education, 2006. 

\bibitem{Krasikov-Roditty}
I. Krasikov, Y. Roditty, \emph{On a reconstruction problem for sequences}, J. Combin. Theory Ser. A, \textbf{77} (1997), 344--348.

\bibitem{Levenstein}
V.I. Levenstein, \emph{Efficient reconstruction of sequences from their subsequences and supersequences}, J. Combin. Theory Ser. A, \textbf{93(2)} (2001), 310--332.

\bibitem{Lothaire}
M. Lothaire, \emph{Combinatorics on words}, Cambridge Mathematical Library, Cambridge University Press (1997).

\bibitem{Osin}
D. Osin, \emph{Acylindrically hyperbolic groups}, Trans. Amer. Math. Soc. 368 (2016), 851-888.

\bibitem{Sapir}
M. Sapir, \emph{Combinatorial Algebra: Syntax and Semantics}, Springer (2014).

\bibitem{Schoenhage-Strassen}
A. Schönhage, V. Strassen, \emph{Schnelle Multiplikation großer Zahlen}, Computing, \textbf{7} (1971), 281--292.

\bibitem{Toom}
A. L. Toom, \emph{The complexity of a scheme of functional elements simulating the multiplication of integers},
Dokl. Akad. Nauk SSSR, 150 (1963), 496--498. (in Russian). English translation in Soviet Mathematics \textbf{3} (1963), 714--716.


\end{thebibliography}
\end{document}